\documentclass[reqno,11pt]{amsart}
\usepackage{latexsym}
\usepackage{amsfonts}
\usepackage{amssymb}
\usepackage{mathrsfs}
\usepackage[all]{xy}
\usepackage{bbm}
\usepackage{xcolor}
\usepackage{ulem}
\usepackage[mathscr]{euscript}

\usepackage[english]{babel}
\usepackage{amsmath}
\usepackage{graphicx}
\usepackage[colorinlistoftodos]{todonotes}
\usepackage[colorlinks=true, allcolors=blue]{hyperref}
\usepackage{tikz-cd}
\usepackage{tikz}
\usepackage{pst-node}
\usepackage{mathtools}
\usetikzlibrary{graphs}
\usepackage{dsfont}

\usepackage[T2A]{fontenc}

\theoremstyle{plain}
\newtheorem{thm}{Theorem}[section]
\newtheorem{claim}{Claim}[thm]
\newtheorem{lem}[thm]{Lemma}

\newtheorem{prop}[thm]{Proposition}
\newtheorem{cor}[thm]{Corollary}

\theoremstyle{definition}
\newtheorem{defn}[thm]{Definition}

\newtheorem{ques}[thm]{Question}

\theoremstyle{remark}

\numberwithin{equation}{section}
\newcommand{\BA}{\mathbf{BA}}
\newcommand{\DI}{\mathbf{DI}}
\newcommand{\SL}{\operatorname{SL}}

\newcommand{\Ad}{\operatorname{Ad}}
\newcommand{\Lie}{\operatorname{Lie}}

\newcommand{\dist}{\operatorname{dist}}

\newcommand\on[1]{\operatorname{#1}}
\newcommand\ca[1]{\mathcal{#1}}

\newcommand\hs{homogeneous space}

\newcommand{\nz}{\smallsetminus\{0\}}

\newcommand{\R}{{\mathbb{R}}}
\newcommand{\Co}{{\mathbb{C}}}
\newcommand{\Z}{{\mathbb{Z}}}
\newcommand{\N}{{\mathbb{N}}}

\newcommand\hd{Hausdorff dimension}

\newcommand{\vp}{{\bf p}}
\newcommand{\vq}{{\bf q}}

\newcommand{\fg}{{\mathfrak{g}}}
\newcommand{\fn}{{\mathfrak{n}}}
\newcommand{\fh}{{\mathfrak{h}}}
\newcommand{\ad}{{\operatorname{ad}}}
\newcommand{\tr}{{\text{tr}}}
\newcommand{\ignore}[1]{{}}
\newcommand{\ggm}{G/\Gamma}

\newcommand {\new}[1]   {\textcolor{brown}{#1}}

\newcommand {\comm}[1]   {\textcolor{red}{#1}}
\newcommand\eq[2]{
\begin{equation}
\label{eq:#1}
{#2}
\end{equation}
}
\newcommand{\equ}[1]{\eqref{eq:#1}}

\usepackage{todonotes}
\usepackage{setspace}
\newcounter{mycomment}

\title[Constructing bounded orbits]{Constructing bounded orbits  of special types \\ on homogeneous spaces}
\author{Manfred Einsiedler}
\author{Dmitry Kleinbock}
\author{Anurag Rao}
\address{ETH, Z\"urich} 
\email{manfred.einsiedler@math.ethz.ch}
\address{Department of Mathematics, Brandeis University, Waltham MA}
\email{kleinboc@brandeis.edu}
\address{SIMIS, Fudan University China} \email{arao@simis.cn}

\subjclass{37A17; 37A25, 37D40, 11J70}
\keywords{homogeneous flows, orbit closures, bounded orbits, mixing}

\date{January 2026}
\thanks{The second-named author was supported by NSF grant DMS-2155111.}
\begin{document}

\begin{abstract}
    Let $X = G/\Gamma$ be a quotient of a real Lie group %with respect to 
    by a non-uniform lattice. Consider a one-parameter subgroup $F$ of $G$ that is $\Ad$-diagonalizable over $\Co$ and whose action 
    %denoted by $F$, with respect to which $\Lie(G)$ is  and with respect to which 
    on $X$ is mixing. In this dynamical system we study the set of points $x \in X$ with a precompact orbit, written as $E(F,\infty)$, which %was proved in \cite{KM} 
    {is known} to be a dense  subset of $X$ of full Hausdorff dimension. We prove that $E(F,\infty)$ is %irreducible 
indecomposable in the following sense: given any
 $y \in E(F,\infty)$, 
the set of $x \in 
 E(F,\infty)$ for which $y \in \overline{F_+x}$, where $F_+$ denotes the positive ray in $F$, is uncountable and dense in $E(F,\infty)$.
    %In fact, the set of points $x$ as above has codimension less than $\dim Z - 1$, where $Z$ denotes the {neutral subgroup of $G$ with respect to $F$}.
    When the dimension of the neutral subgroup of $G$ with respect to $F$ is $1$ we demonstrate, for any $\varepsilon>0$, the existence of many points $x \in X$ whose  orbit closure $\overline{F_+x} \subset X$ is compact and has Hausdorff dimension at least  $\dim X - \varepsilon$.
\end{abstract}

%\subjclass{11J04; 11J13, 37A17, 37D40}
\maketitle
%\large

%General dynamical systems
\section{Introduction}\label{intro}

 Let $X$ be a   metric  space, and let $F$ be {either of the (semi) groups $\R$ or $\Z$ or $\R_{\ge 0}$ or $\Z_{\ge 0}$ represented as a family of} self-maps of $X$. We will write $F = \{g_t\}$. For $g\in F$ we will denote its action on $X$ by 
$(g,x) \mapsto g%\cdot 
x$.
 Now fix a subset $Y\subset  X$ and define the set of points of $X$ {\sl approaching} $Y$ by $F$ as
    \begin{equation*}
        A(F,Y) := \big\{ x \in X : \exists \text{ an unbounded sequence }(g_{t_k})_{k \in \N} \subset F \text{ with }
        {\lim_{k \to \infty} g_{t_k} %\cdot 
        x \in Y}
 \big\}.
    \end{equation*}
    Its complement,  the set of points of $X$ {\sl escaping} {$Y$} by $F$ will be denoted by $E(F,{Y})$, namely
    \begin{equation*}
        E(F,{Y}) := \bigcap_{y \in Y}\big\{ x\in X : y%\cap  \operatorname{closure}
        \notin\overline{(F\smallsetminus F_0) x} 
        %= \varnothing} %\cdot 
         \text{ for some bounded }F_0\subset F\big\}.
    \end{equation*}
 If $X$ is %noncompact
not bounded, we make the convention that $x\in A(F,\infty)$ when the orbit $F%\cdot 
 x$ is unbounded.

 {Take $y\in Y$ and   consider $A(F,y):= A(F,\{y\})$ and $E(F,y):= E(F,\{y\})$.}  A natural question one could ask  is how large the sets $A(F,y)$   and   $E(F,y)$ can be. If the action admits an ergodic invariant probability measure, or has some hyperbolic behavior, some instances of this question can be  answered. Namely, under the assumption of existence of an ergodic measure $\mu$ of full support the set  $A(F,y)$ has full measure for any $y\in X\cup \{\infty\}$. Additionally,  for many hyperbolic systems one can prove  that the $\mu$-null sets $E(F,y)$  are \textsl{thick},  that is, their intersection with any nonempty open set has full Hausdorff dimension.  For results of this kind see e.g.\ \cite{Da3} for toral endomorphisms, \cite{U, Do} for Anosov flows, and {\cite{Da1, KM, Kl2, AGK}} for partially hyperbolic flows on \hs.
Furthermore, in many cases these sets can be proved to be winning in the sense of Schmidt games \cite{S};  see e.g.\  \cite{AGK, BFK,Da1,Da3,Ts1,Ts2,wu1,wu2}).
Consequently in these cases  it follows that both $A(F,y)$   and   $E(F,y)$ are thick.
%, that is, their intersection with any non-empty open set has full Hausdorff dimension. 
Moreover, both full-measure and winning conditions are stable with respect to countable intersections; thus
%. In particular, it implies that 
for any choice of countably many semigroups $F_i$ and points $y_i \in X\cup \{\infty\}$  for which the above conclusions can be established, it holds that both $\bigcap_i A(F_i,y_i)$   and   $\bigcap_i E(F_i,y_i)$  are thick. 

\smallskip
This paper addresses the following question: what if one considers 
a mixed case, that is, investigate the intersection \eq{mixed}{A(F,y) \cap E(F,z),} where $y,z\in X\cup \{\infty\}$? Problems of this type are amenable neither to the full-measure argument, nor to the technique based on Schmidt games. Also it is clear that some conditions on $y,z$ should be imposed; for example if $y=z$ it is clear that the intersection  \equ{mixed} is empty. 

\smallskip
In what follows we restrict our attention to $F\cong \R$, %and write $F = \{g_t: t\in\R\}$, 
also denoting $F_+ := \{g_t: t\ge 0\}$.
%We will suppress $F$ from our notation, referring to $A(F,y)$ and  $E(F,y)$ as to $A(y)$ and  $E(y)$ respectively.  We will also similarly abbreviate the notation for semigroup actions, letting $A_\pm(y):= A_{F_\pm}(y)$ and $E_\pm(y):= E_{F_\pm}(y)$. 
Assume that the $F$-action on $X$ is continuous. Then let us record {a natural} 
%obvious 
obstruction to  
the set  \eq{mixedplus}{A(F_+,y) \cap E(F_+,z)} being non-empty:  
%if $y \in A(z)$ (that is, $z$ can be approached by the trajectory of $y$ either in the positive or in the negative time direction), then $A(F_+,y) \cap E_+(z) = \varnothing$.   We state this observation in the form of

\begin{lem}\label{lem: obstruction}
    If {$y \in X$  and $z \in X \cup \{\infty\}$} are such that $y \in A(F,z)$ (that is, $z$ can be approached by the trajectory of $y$ either in the positive or in the negative time direction), then 
    \begin{equation*}
        A(F_+,y) \subset A(F_+,z).
    \end{equation*}
\end{lem}

\begin{proof}
Let $x \in A(F_+,y)$, and let $(s_\ell)_{\ell \in \N}$ be an unbounded %positive 
sequence in $\R_{\geq 0}$ with
\begin{equation*}
    \lim_{\ell \to \infty} g_{s_\ell}x = y.
\end{equation*}
Let $W \subset X$ be an open neighborhood of $z$. In the case, when $z = \infty$, we take $W$ to be an open set with %compact
bounded complement. %in $X$.

By the hypothesis, there exists a time $t \in \R$ for which $g_t y \in W$. By the continuity of the action we have
\begin{equation*}
    \lim_{\ell \to \infty} g_{s_\ell}g_tx = g_t y \in W.
\end{equation*}
Note that the sequence $(s_\ell + t)_{\ell \in \N}$ is unbounded and eventually positive. Since $W$ was arbitrary, we conclude that $x \in A(F_+, z)$.
\end{proof}
\ignore{{Suppose} we have a point $x \in A(F_+,y) \cap E(F_+,z)$. By {the} hypothesis, take an unbounded sequence $(t_k)_{k \in \N} \subset \R$ with
\begin{equation}\label{eq: y_0 to y_1}
    \lim_{k \to \infty} g_{t_k} y= z
\end{equation}
(in the case $z = \infty$ this means that $g_{t_k} y$ leaves every compact subset of $X$).
Since $x$ is an element of $E(F_+,z)$, we can take an open neighborhood $W$ containing $z$ (if $z = \infty$ we can let $W$ be the complement of a compact subset of $X$) and   $t_0>0$ such that
\begin{equation}\label{eq: tx notin W}
    g_tx \notin W \text{  for all }t> t_0.
\end{equation}
Using the fact that $x \in A(F_+,y)$, we take an unbounded sequence $(s_k)_{k \in \N} \subset \R_{\geq 0}$ with
\begin{equation}\label{eq: x to y_0}
    \lim_{k \to \infty} g_{s_k} x = y.
\end{equation}
The convergence in \eqref{eq: y_0 to y_1} guarantees us some $p \in \N$ with 
\begin{equation*}
    {g_{t_{p}} %\cdot 
    y \in W}.
\end{equation*}
The continuity of the action and the convergence in \eqref{eq: x to y_0} guarantees us that 
\begin{equation*}
    \lim_{k \to \infty} g_{t_p + s_k} x = g_{t_p}%\cdot 
    y \in W. 
\end{equation*}
However, for $k$ large enough, $t_p + s_k$ will be larger than $t_0$, contradicting %the condition 
\eqref{eq: tx notin W}.}

{In view of the above lemma it makes sense to look for conditions guaranteeing that the set  \equ{mixedplus} is reasonably large 
%one can easily shows that 
%Question \ref{ques: bad minus dirichlet} 
under the assumption that $y \in E(F,z)$. 
%In other words, the question %\ref{ques: bad minus dirichlet} 
%is really about dynamics on the flow-invariant null subset $Y :=  E(z)$ of $X$. Namely, one  takes $y\in Y$ and aims to find  many points in $Y$ approaching $y_0$.} i
In the present paper we will  do it in the following set-up: we let $X = G/\Gamma$, where $G$ is  a real %connected 
Lie group 
%with {discrete center}  
and $\Gamma \subset G$ is a %\new{non-%uniform} 
lattice,
% which is {not co-compact}, 
%\new{take $y_1 = \infty$}, 
and    consider the $\R$-action on $X$ via the left multiplication by a one-parameter subgroup of $G$:
    \begin{equation}\label{gt}
        \big\{ g_t := \exp(ta_0) \in G: t \in \R \big\},
    \end{equation}
    where ${a_0}$ is an %ad-diagonalizable 
    element of %the Lie algebra 
    {$\fg:=\Lie(G)$, the Lie algebra of $G$}. 
    We shall fix once and for all an inner product on $\fg$ to induce a right-invariant Riemannian metric $\dist_G$ on $G$.
 Note that  the invariance of the metric implies that it is geodesically complete, and so the Hopf--Rinow theorem implies that it is proper (closed and bounded sets are compact). Then one obtains a metric on $X = G/\Gamma$ by setting
\begin{equation*}
    \dist_X(g_1\Gamma, g_2\Gamma) := \inf\big\{\dist_G(g_1\gamma, g_2): \gamma \in \Gamma \big\}.
\end{equation*}
This is the set-up in which we will attempt to study sets of the form  \equ{mixedplus}. We will denote by $m_X$  the $G$-invariant Haar probability measure on $X$, and by $m_G$ the %corresponding 
Haar measure on $G$ scaled so that the measure of any fundamental domain for $\Gamma$ is equal to $1$. Further,
    %Also 
   {in this paper} we will %always 
   assume that $X$ is noncompact and $z = \infty$; that is, we will seek to construct bounded $F^+$-trajectories with prescribed limit points.  
    
  Recently the second- and third-named authors proved the following result. Take \linebreak  $G = \SL_2(\R)$, $\Gamma = \SL_2(\Z)$ and $X = G/\Gamma$. For $t,\alpha\in\R$ consider 
\begin{equation*}\label{eq: gtsl2}
        g_t := \left[ {\begin{array}{cc} e^t & 0 \\ 
    0 & e^{-t}  \end{array}}\right]\in G,\ F = \{g_t: t\in\R\},
    \end{equation*}
    and  %\eq{xalpha}
$${x_\alpha  := \left[ {\begin{array}{cc} 1 & \alpha \\ 
    0 & 1  \end{array}}\right]\Gamma\in X.} 
    \ignore{the latter being identified with the lattice in $\R^2$ generated by $\left[ {\begin{array}{cc} 1 \\ 
    0  \end{array}}\right]$ and $ \left[ {\begin{array}{cc}  \alpha \\ 
    1 \end{array}}\right]$.}$$
    
    \begin{thm}\label{main-lemma} \cite[Theorem 1.3]{KR1}
Let $X$ and $F$ be as above, and let $y \in E(F,\infty)$; then
\begin{equation}\label{dim1}
     \dim\big(\{ \alpha \in \R : x_\alpha \in  A(F_+,y) \cap E(F_+,\infty)\}\big) = 1.
\end{equation}
%has Hausdorff dimension $1$. 
\end{thm}

%This notion of dimension is independent of the choice of such a metric
Here and hereafter "dim" stands for \hd.
 In fact, even though it was not mentioned in \cite{KR1}, the proof actually demonstrates the thickness of the set in the left-hand side of \eqref{dim1}. Also,  by the standard "slicing" argument (see %section 
\S \ref{sec: reduction to H}), it can be easily derived from Theorem \ref{main-lemma} 
that  for any $y \in E(F,\infty)$ the set $A(F_+,y) \cap E(F_+,\infty)$ %has full \hd
is thick.
The proof of the above theorem used the technique of continued fractions, and in fact the result can be interpreted within the framework of Diophantine approximation. Namely it constructs numbers that are badly approximable but not $\nu$-Dirichlet improvable, where $\nu$ is an irreducible norm on $\R^2$ whose unit ball is not a parallelogram. See \cite{AD, KR, KR1, KRS} for more detail, and \cite{MS} for some new results in that direction.  
%We elaborate on that further  in \S \ref{dioph} below. 

\smallskip
One of the main results of this paper is a generalization of Theorem \ref{main-lemma} to a wider class of flows on \hs s.

\begin{thm}\label{thm: main theorem}
    Let $X = G/\Gamma$  and $F = \{g_t : t\in\R\}\subset G$, where 
    \begin{itemize}
    \item $G$ is  a %real connected semi-simple 
    Lie group with discrete center,
    %with finite center and with no compact factors, 
    \item $\Gamma$ is a non-uniform 
    lattice in $G$,
    %\item $m_X$ is the Haar probability on $X$,
\item $g_t$ is as in \eqref{gt}, where $a_0 \in \fg\nz$ is $\ad$-diagonalizable over $\Co$;
 \item the action of $F$  on $(X,m_X)$ is mixing.
 %, where $m_X$ is the Haar probability measure on $X$.
\end{itemize}
    Let $y \in E(F,\infty)$. Then for any non-empty open $U\subset X$ we have 
    \begin{equation}\label{eq: dim X - dim Z + 1}
        \dim \big(U\cap A(F_+,y) \cap E(F_+,\infty)\big) \geq \dim X - \dim Z + 1,
    \end{equation}
    where $Z \subset G$ is the {neutral subgroup of $G$ with respect to $F$} . 
    %On the other hand, if $y_0 \notin E(\infty)$, then
    %\begin{equation*}
    %    A_+(y_0) \cap E_+(\infty) = \varnothing.
   % \end{equation*}
\end{thm}
See Definition \ref{defn: expanding, contracting and neutral groups} below for a precise definition of $Z$.
In view of Moore's Theorem \cite{Mo}, the mixing assumption is satisfied when $G$ is a connected semi-simple 
    Lie group
    with finite center and with no compact factors, and $\Gamma\subset G$ is irreducible. However mixing is not restricted to the semi-simple case, see e.g.\  \cite{BM, Kl}.  
%In particular, 
Note that when $G =  \SL_2(\R)$ (that is, under the assumptions of Theorem \ref{main-lemma}) we know that {$\dim Z=\dim F = 1$}, and thus the set $A(F_+,y) \cap E(F_+,\infty)$ is thick. Moreover, in the latter case one can generalize Theorem~\ref{main-lemma}  as follows.

  \begin{thm}\label{thm: dense-orbits}  
%$G =  \SL_2(\R)$, $\Gamma$ a  {nonuniform} lattice in $G$, $X=\ggm$, and let $F$ be as in \equ{gtsl2}. 
Let $G$, $\Gamma$ and $F$ be as in Theorem \ref{thm: main theorem}, and assume in addition that \linebreak {$\dim Z = %\dim F$
1$}.
Then for any {compact} 
%compact, 
F-invariant
subset $B$ of $E(F,\infty)$ there exists a thick set of $x \in E(F_+,\infty)$ such that the closure of the (bounded) trajectory $F_+x$ contains $B$.
%has Hausdorff dimension $1$.
 \end{thm}

In addition to $G =  \SL_2(\R)$, the assumptions of the above theorem hold when \linebreak $G =  \SL_2(\R)\ltimes \R^2$; the choice $\Gamma =  \SL_2(\Z)\ltimes \Z^2$ yields the space of two-dimensional unimodular grids, cf.\ \cite{Kl}.
%Combining this with a recent result of the first named author and Mirzadeh \cite{KMi} on the dimension drop phenomenon, one obtains
  \begin{cor}\label{cor: dim-drop}  
Let $G$, $\Gamma$, $X$ and $F$ be as in Theorem \ref{thm: dense-orbits}; then for any $\varepsilon > 0$ there exists a thick set of $x\in X$ such that the orbit closure $\overline{F^+x}$ is compact and has Hausdorff dimension greater than $\dim X - \varepsilon$.
 \end{cor}
 \begin{proof}
 When $%G, \Gamma, 
X$ and $F$ are as in Theorem \ref{thm: main theorem}, it follows from %the proof of 
 \cite[Theorem 1.1]{KM} that, given $\varepsilon>0$, one can find a large enough compact set $K \subset X$ such that the compact $F$-invariant set
 \begin{equation*}
     B = \{x \in X: Fx \subset K\}
 \end{equation*}
 has Hausdorff dimension greater than $\dim X - \varepsilon$.
Under the additional hypothesis of $\dim Z = 1$, an application of Theorem \ref{thm: dense-orbits} to $B$ then gives the result.
 \end{proof}

 \noindent{\bf Acknowledgements.}
    The second-named author is grateful to ETH (Z\"urich) for its hospitality during a sabbatical stay in 2022, and to Barak Weiss for helpful discussions. The third-named author thanks N.\ Chandgotia, T.\ Mesikepp and Y.\ Peres for holding public-access office hours at TIFR-CAM, Peking University and Tsinghua University respectively.

 \section{Reduction to the expanding subgroup}\label{sec: reduction to H}
 In this and the remaining part of the paper we let $G$, $\Gamma$  and $F$ be as in Theorem \ref{thm: main theorem}. Recall that $F = \{g_t =\exp(ta_0): t\in\R\}$, where 
   ${a_0}\in \fg\nz$ is    $\ad$-diagonalizable over $\mathbb{C}$.
   \begin{defn}[Expanding, contracting and neutral subgroups]\label{defn: expanding, contracting and neutral groups} Consider 
\begin{equation*}
   \begin{split}
       \fh &:= \left\lbrace b \in \fg: \lim_{t \to -\infty} \Ad(g_t)b = 0\right\rbrace, \\
         \fh^- &:= \left\lbrace b \in \fg: \lim_{t \to \infty} \Ad(g_t)b = 0\right\rbrace,\\
         \fn &:= \left\lbrace b \in \fg: \sup_{t \in \R} \|\Ad(g_t)b\| < \infty\right\rbrace, 
    \end{split}
   \end{equation*}
   and let $H$, $H^-$ and $Z$ be the connected Lie subgroups of $G$ associated to the Lie algebras $\fh$, $\fh^-$ and $\fn$ respectively. From the diagonalizability assumption, we have 
   \begin{equation}\label{eq: real decomposition}
     \fg = \fh \oplus \fn \oplus \fh^-,
 \end{equation}
   hence $G$ is locally (in a small neighborhood of identity) a direct product of $H$, $H^-$ and $Z$.   In fact, if one considers the subalgebras of $\fg \otimes_\R \Co$ which are direct sums of eigenspaces for $\ad_{a_0}\otimes 1$ corresponding to eigenvalues with positive, zero and negative real part respectively, then one obtains the above Lie algebra decomposition \eqref{eq: real decomposition} simply by intersecting with $\fg$. \end{defn}
We remark that   $\fn$ is a proper subalgebra of $\fg$, otherwise $\Ad(g_t)$ would belong to a compact group and the mixing hypothesis would fail.
    %if it were all of $\fg$ (cf. Proposition \ref{thm: KM mixing}), 
    Moreover, since $G$ admits a lattice,
    %$\Gamma$, 
    the modular function of $G$ is constant. Thus ${\ad_{a_0}}$ has trace $0$, and, therefore, both $\fh$ and $\fh^-$ (hence $H$ and $H^-$ as well) are non-trivial.
\begin{defn}[Measures on $H$ and $\tilde H$]\label{horosph}
   We fix $m_H$, a Haar measure on $H$. We further consider the connected subgroup $\tilde{H}$ associated to the Lie algebra $\fn \oplus \fh^-$ and a left Haar measure $m_{\tilde{H}}$ on $\tilde{H}$, scaled so that $m_G$ is a direct product of $m_H$ and $m_{\tilde{H}}$, see \cite[Ch.~VII, \S 9, Proposition 13]{B} or \cite[Theorem 8.32]{Kn} for justification.
\end{defn}

    %(See \cite[Proposition 1.91]{Kn} if required.)
By ordering the eigenspaces of $\ad_{a_0}\otimes 1$ with respect to their real parts, we see that $H$ is nilpotent and hence unimodular. Thus, $m_H$ is a bi-invariant Haar measure on $H$. For future reference, we also note that $H$ is %indeed 
    simply connected. Indeed, note that $\Ad|_{H} : H \to \Ad(H)$ is a covering map. By choosing a basis compatible with the decomposition in \eqref{eq: real decomposition}, $\Ad(H)$ can be realized as a subgroup of {the group} of upper-triangular matrices. Thus $H$ covers a simply-connected space (see \cite[Corollary 1.134]{Kn}) and must itself be simply connected.
    
    %The inner product on $\fg$ induces a norm which we denote by $\|\cdot \|$.
    %Further, u

    %\new{Considerations with exponential maps give the following well-known result.}

 \begin{defn}[Conjugation %on $H$, $H^-$ and $Z$
 {by $g_t$}]\label{defn: congujation}
    We use the notation, for each $t \in \R$,
    \begin{equation*}
        \Phi_t:G \to G;\ \Phi_t(g) := g_t gg_{-t}.
    \end{equation*}
    %Moreover
    In view of the commutation relation %in $\GL(\fg)$
    %\begin{equation*}
        $\Ad\big(\exp(b)\big) = e^{\ad_{b}}$  for all  $b \in \fg$,
   % \end{equation*}
    %If $b \in \fg$ and $t \in \R$, 
    we have
    \begin{equation*}\label{conj}
        %g_t \exp (b)g_{-t}
        \Phi_t\big(\exp (b)\big)= \exp(e^{t\ad_{a_0}}b)\ \ \text{for all $b \in \fg$ and }t \in \R.
    \end{equation*}
    In particular,   $H, H^-$ and $Z$ are invariant under $\Phi_t$, and the latter  %conjugation by $g_t$.}
%In view of %Lemma \ref{lem: contracting expanding explicit computation}, 
  % \eqref{conj}, $\Phi_t$ 
  restricts to a smooth  isomorphism {of} %$H$, $H^-$ and $Z$
  these groups.
   %Similarly $\Phi_t$ restricts to a map on $L$. 
   We can thus define %he %modular 
   %function $\chi: \R \to \R_{>0}$ to be 
   the Radon--Nikodym derivative
    \begin{equation}\label{eq: chi t}
        \frac{dm_H}{d(\Phi_t)_*m_H} = e^{\chi t},\text{ where }\chi := \tr\,\ad_{a_0}|_{\fh}.
    \end{equation}
    The equality above follows since 
    %Since, for any $s,t \in \R$ we have $\Phi_t \circ \Phi_s = \Phi_{(t+s)}$, we see that $\chi $ is actually a {group homomorphism}. %of groups. 
    %In fact, one can see that {$\new{e^{\chi t}} = e^{(\tr\ \ad_{a_0}|_{\fh})t}$}; note that 
    $\exp: \fh \to H$ is a diffeomorphism and $m_H$ is the  pushforward of a translation invariant measure on $\fh$ (see the proof of \cite[Theorem 1.127]{Kn}). Note that $\chi >0$.
    \end{defn}

In what follows for a closed subgroup $P$ of $G$ we  use the  inner product on $\Lie(P)$ induced from $\fg$ to obtain the right-invariant Riemannian distance function $\dist_P$ on $P$.  
%Using the induced inner product on $\fh$, $\fh^-$ and $\fn$ we obtain right-invariant Riemannian distance functions on $H$, $H^-$ and $Z$  denoted by $\dist_H$, $\dist_{H^-}$ and $\dist_Z$. 
For $g\in P$, $x \in X$ %, $a \in \fg$ 
and $r>0$ we use the notation
    \begin{equation*}
        B_P(g, r)%, \ 
        \text{ and } B_X(x,r)%B_\fg(a,r)
    \end{equation*}
    to denote balls in $P$ with respect to $\dist_P$ and balls in $X$ with respect to $\dist_X$.
    %and balls in $\fg$ with respect to the {norm induced by the} inner product.  
    When the base-point is the identity   $e_G\in G$, %or the zero element in $\fg$, 
    we use the notation $B_P(r):= B_P(e_G,r)$.
    % or $B_\fg(r)$. 

     The following lemma summarizes the expanding, contracting and neutral properties of {$\Phi_t$ restricted to} the subgroups $H$, $H^-$ and $Z$. It easily follows from Definition \ref{defn: expanding, contracting and neutral groups} and the  locally bi-Lipschitz property of the exponential map $\fg\to G$.
     
    \begin{lem}\label{lem: contracting expanding explicit computation}
    If $h \in H$ and $h^- \in H^-$, we have
    \begin{equation*}
        \lim_{t\to -\infty} {\Phi_t(h)} = e_G \ \text{ and } \ \lim_{t\to \infty} {\Phi_t(h^-)}  = e_G.
    \end{equation*}
   {Also} there exist constants $C, {r_0} > 0$ such that
    \begin{equation*}\label{eq: Z conjugate distances}
       { \dist_Z\big({\Phi_t(z)},e_G\big) \leq C\dist_Z(z,e_G)}.
    \end{equation*}
     for all $z \in B_Z({r_0})$ and $t \in \R$.
    \end{lem}
    \ignore{\begin{proof}
    %The first claim follows from the commutation relation in $\GL(\fg)$
    %\begin{equation*}
    %    \Ad\big(\exp(a)\big) = e^{\ad_{a}} \ \text{ for all } a \in \fg.
    %\end{equation*}
    %(See \cite[Proposition 1.91]{Kn} if required.)
   % The second claim follows since the exponential maps on $\fh$ and $\fh^-$ are diffeomorphisms to $H$ and $H^-$ respectively. (See \cite[Theorem 1.127]{Kn} if required.)
    %We turn to the third claim.
    Suppose we have $r>0$ and $\theta \in \R$ for which $\lambda = re^{i\theta}$ and $\bar{\lambda} = re^{-i\theta}$ are two distinct eigenvalues for $\ad_{a_0}\otimes 1$ with $\Re \lambda >0$.
    Suppose we have $v, w \in \fg$ such that $v\otimes 1 + w \otimes i \in V_\lambda$ is an eigenvector. Then $v\otimes1 - w \otimes i \in V_{\bar{\lambda}}$ is also an eigenvector. 
    We see that
    \begin{equation*}
        v\otimes 1, w \otimes 1 \in \left(V_{\lambda} \oplus V_{\bar{\lambda}}\right) \cap \fg
    \end{equation*}
    forms an $\R$-independent set and that the $\R$-linear mapping $\varphi :\Co  \to \left(V_{\lambda} \oplus V_{\bar{\lambda}}\right )\cap \fg$ given by
    \begin{equation*}
        1 \mapsto v \otimes 1,\ i \mapsto w \otimes 1
    \end{equation*}
    makes the following diagram commute
    \begin{equation}
    \begin{tikzcd}
    \Co \arrow[r, "\varphi"] \arrow[d, "{re^{-i\theta}}"] & \left(V_\lambda \oplus V_{\bar{\lambda}}\right) \cap \fg \arrow[d, "\ad_{a_0}"]   \\
    \Co \arrow[r, "\varphi"] &  \left(V_{\lambda} \oplus V_{\bar{\lambda}}\right) \cap \fg
    \end{tikzcd}
    \end{equation}
    Thus, if $b = \varphi(\zeta)$ lies in the image of this map, we see that 
    \begin{equation}\label{eq: complex embedding computation}
    \begin{split}
        g_t \exp(b) g_{-t} &= \exp(e^{t\ad_{a_0}}b) \\
        &= \exp(e^{t\ad_{a_0}}\varphi(\zeta)) \\
        &= \exp\circ\varphi \left(e^{tre^{-i\theta}}\zeta\right) \\
        &= \exp(e^{tr\cos \theta}\varphi(e^{-itr\sin \theta}\zeta)).
    \end{split}
    \end{equation}
    Since {$\Re \lambda = r\cos \theta >0$}, it follows that
    \begin{equation*}
        \lim_{t \to - \infty} \varphi(e^{tre^{-i\theta}}\zeta) = 0.
    \end{equation*}
    In general vectors in $\fh$ can be written as the sum of {eigenvectors} and vectors coming from embeddings $\varphi$ as above.  Thus, linearity shows that for every $b \in \fh$, we have
    \begin{equation*}
        \lim_{t\to -\infty} g_t \exp(b) g_{-t} = e_G.
    \end{equation*}
    Since $\exp: \fh \to H$ is a diffeomorphism, we are done. 
    An entirely analogous argument shows that $\lim_{t \to \infty} g_th^-g_{-t} = e_G$ for every $h^- \in H^-$. 
  %  \begin{equation*}
   %     \lim_{t \to \infty} g_th^-g_{-t} = e_G.
  %  \end{equation*}
    {For the last part of the lemma we consider a neighborhood of the identity in $Z$ which is a diffeomorphic image of a neighborhood of $0$ in $\fn$ and for which the exponential map is bi-Lipschitz. Then one computes as in \eqref{eq: complex embedding computation} while noting that $r\cos \theta = 0$ for each of the maps $\varphi$}.
    \end{proof}}

       The theorems in the previous section will be reduced to the following.
    \begin{thm}\label{thm: theorem on H}
        Let $X = G/\Gamma$ and $F = \{g_t \in G: t \in \R\}$ be as in Theorem \ref{thm: main theorem}. Let $\rho >0$, $x_0 \in X$, and let $B$ be a countable set with ${FB} \subset X$ bounded; in particular, $B \subset E(F,\infty)$. Then %there exists an 
        {the set of $h \in H$ with the properties
        \begin{itemize}
            \item $hx_0 \in E(F_+, \infty)$
            \item for every $y \in B$, $hx_0 \in A\big(F_+, B_Z(\rho )y\big)$ 
        \end{itemize}
        %In fact, the set of elements $h$ with the above properties 
        is thick in %$(H, \dist_H)$
        $H$}.
    \end{thm}
    We now give the
    \begin{proof}[Proof of Theorem \ref{thm: main theorem} from Theorem \ref{thm: theorem on H}]
        Let $U \subset X$ be open, and let $y \in E(F,\infty)$. Fix $x_0 \in U$, and let $r >0$ be such that the map 
        \begin{equation*}
            \varphi: B_Z(r) \times B_{H^-}(r) \times B_{H}(r) \to U;\ \varphi(z,h^-,h) = zh^-h x_0
        \end{equation*}
        is a diffeomorphism onto an open subset of $U$. 
        It follows from the definition of $\fn$ that there exists constants $\rho_1, C >0$ such that for any $\rho < \rho_1$ and any unbounded sequence $(t_n)_{n \in \N} \subset \R$ we have a subsequence $(t_{n_k})_{k \in \N}$ such that the sequence of maps $\Phi_{t_{k_n}}|_{B_Z(C\rho)}$ converges pointwise to some smooth immersion
        \begin{equation}\label{eq: look at me}
        {\begin{aligned}
        \Phi : B_Z(C\rho) \to Z
             \end{aligned}}
        \end{equation}
        whose image contains the ball $B_Z(\rho)$.
        Apply Theorem \ref{thm: theorem on H} with $\rho < \rho_1$. Further assume that $C\rho < r$.
         \ignore{Take $r_0,C$  as in Lemma \ref{lem: contracting expanding explicit computation}. By shrinking  $r$ if needed we can assume that $r < r_0$ and that $r/C < \rho_2$. 
         Apply Theorem \ref{thm: theorem on H} with $\rho = r/C$ 
        and %with 
        $B = \{y\}$.}
        This gives us a set %full dimension set of points $S \subset B_H(r)$ such that every $h \in S$ has
        \begin{equation*}
          S := \left\{h\in B_H(r):   hx_0 \in E(F_+, \infty) \cap A\big(F_+,B_Z(\rho)y\big)\right\}
        \end{equation*}
        with $\dim S = \dim H$.
      Then 
        for any $s \in \R$, %$z \in B_Z(r)$ and $h^- \in B_{H^-}(r)$ 
        {$z \in Z$ and $h^- \in H^-$} compute
    \begin{equation}\label{eq: bounded invariance}
        \begin{split}
            g_t g_sz h^- hx_0 &= g_s{\Phi_t(z)}\Phi_t(h^-)g_th x_0.
        \end{split}
        \end{equation}
        Lemma \ref{lem: contracting expanding explicit computation} thus shows that $g_szh^- S x_0 \subset E(F_+, \infty)$.
        Further, let $h \in S$, and {suppose that} $hx_0 \in A(F_+,zy)$ with $z \in B_Z({\rho})$. 
        This gives us an unbounded sequence {$(t_n)_{n \in \N} \subset \R_{\geq 0}$ with
        \begin{equation*}
            \lim_{n \to \infty} g_{t_n}h x_0 = z y.
        \end{equation*}
        %By the assumption on $r$, we have a pointwise limit $\Phi$ of the sequence of maps $\{\Phi_{t_k} : B_Z(r) \to Z\}$.
    Choose a subsequence  $(t_{n_k})$ and map $\Phi$ as in \eqref{eq: look at me}, and let $w = \Phi^{-1}(z^{-1}) \in B_Z(C\rho)$.
        %Let $w \in B_Z(r)$ have the property that $\Phi(w) =z^{-1} \in B_Z(\alpha r)$.
        We then see that for any $h^- \in H^-$ and $s \in \R$
    \begin{equation}\label{eq: approach invariance}
        %\begin{split}
            \lim_{k \to \infty }g_{t_{n_k}-s} \big(g_s  w h^- h\big) x_0 = \lim_{k \to \infty}  \Phi_{t_{n_k}}(w) \Phi_{t_k}(h^-) g_{t_{n_k}}hx_0 %
           % \\
           % &
           =  z^{-1}zy%\\
            %&
            =  y.
        %\end{split}
        \end{equation}}
        {The %se two
         computations  in }\eqref{eq: bounded invariance} and \eqref{eq: approach invariance} show the following. If $h \in S$, then there exists some $w \in B_Z(r)$ such that for all $s \in \R$ with $g_s w \in B_Z(r)$ and all $h^- \in B_{H^-}(r)$ we have
        \begin{equation*}
            g_swh^-hx_0 \in E(F_+, \infty) \cap A(F_+, y).
        \end{equation*}
        Thus, if we define %$T$ \subset B_Z(r) \times B_{H^-}(r) \times B_{H}(r)$ as 
        \begin{equation*}
            T := \left\lbrace (z,h^-,h)\in B_Z(r) \times B_{H^-}(r) \times B_{H}(r): zh^-h x_0 \in E(F_+,\infty) \cap A(F_+,y)\right\rbrace
        \end{equation*}
         and 
         let $\pi: B_Z(r) \times B_{H^-}(r) \times B_{H}(r) \to B_{H}(r)$   be the projection onto the third factor, then in the commutative diagram
         \begin{equation*}
    \begin{tikzcd}
    \pi^{-1}(S) \cap T \arrow[r] \arrow[d, "\pi"] & B_{Z}(r) \times B_{H^-}(r) \times B_{H}(r) \arrow[d, "\pi"]   \\
    S\arrow[r]  & B_{H}(r)
\end{tikzcd}
\end{equation*}
we have 
%\comm{This is still hard to understand. How does the dimension estimate follow from the diagram? perhaps you can write that a certain set is a subset of another set, and them bound the Hausdorff dimension?}
\begin{equation*}
    \dim\big(\pi^{-1}(h) \cap T\big) \geq 1 + \dim H^- \text{ for all } h \in S.
\end{equation*}
We now record the following %Fubini theorem.
dimension estimate.
\begin{lem}[Lemma 1.2.1 in \cite{Kl2}]\label{lem: Eilenberg}
  Let $M_1, M_2$ be Riemannian manifolds, and let $\pi: M_1\times M_2 \to M_2$ be %a Lipschitz map between them. 
 the natural projection.  Let $E \subset M_1\times M_2$. Then 
    \begin{equation*}
        \dim E \geq \dim \pi(E) + \inf_{p \in \pi(E)} \dim \left(\pi^{-1}(p) \cap E \right).
    \end{equation*}
\ignore{    Let $M_1, M_2$ be metric spaces and let $\pi: M_1 \to M_2$ be a Lipschitz map between them. 
    Let $E \subset M_1$. Then we have the dimension estimate:
    \begin{equation*}
        \dim E \geq \dim \pi(E) + \inf_{p \in \pi(E)} \dim \left(\pi^{-1}(p) \cap E \right).
    \end{equation*}}
\end{lem}
\begin{proof}%[Proof of Lemma \ref{lem: Eilenberg}]
The use of a covering and a bi-Lipschitz equivalence reduces the statement to  $M_1$ and  $M_2$ being open subsets in
Euclidean spaces, in which case the estimate is given by a theorem of Marstrand, see \cite[Theorem  1]{Ma} or \cite[Theorem
5.8]{F}\footnote{In the aforementioned references the theorems is stated and proved for the case $\dim M_1 = \dim M_2 = 1$, but the same proof works for arbitrary dimensions.}. See also    \cite[Theorem 1.1]{EH} for a modern treatment.
\ignore{(see also \cite[Theorem 7.7]{Ma}). This Theorem states that, for any $0 \leq t \leq s< \infty$, there exists a constant $C>0$ such that
    \begin{equation}\label{eq: upper integral}
        \int_{M_2}^* \ca{H}^{s-t}(E \cap \pi^{-1}(h)) d\ca{H}^t(h) \leq C \ca{H}^s(E).
    \end{equation}
  Here $\ca{H}^*$ denotes the corresponding Hausdorff measure and the integral on the left is the upper integral. Choosing
    \begin{equation*}
       t < \dim \pi(E)\ \text{ and }\ s = t + \inf_{p \in \pi(E)} \dim (\pi^{-1}(p) \cap E)
    \end{equation*}
    shows that the right hand side of \eqref{eq: upper integral} is infinite and gives the result.}
\end{proof}

Now let us apply the above lemma with $M_2 = B_H(r)$, $M_1 = 
B_Z(r) \times B_{H^-}(r)$ and 
$E =  \pi^{-1}(S) \cap T$.
%and $E= M_1$. 
Note that $\pi(E) = S$. We thus obtain
\begin{equation*}
\begin{split}
    \dim T &\geq \dim S + 1 + \dim H^-  \\
    &= \dim H + 1 + \dim H^- %\\
    %&
    = \dim X - \dim Z + 1.
    \end{split}
\end{equation*}
Since $T$ is mapped to $U$ by $\varphi$, we have 
\begin{equation*}
\dim \big( U \cap A(F_+,y) \cap E(F_+,\infty) \big) \geq \dim X - \dim Z + 1,
\end{equation*}
which completes the proof.
    \end{proof}

We also have
\begin{proof}[Proof of Theorem \ref{thm: dense-orbits} from Theorem \ref{thm: theorem on H}]
    By the hypothesis we have a compact, $F$-invariant  subset $B \subset E(F,\infty)$. We are to construct points $x$ with $B \subset %\text{closure}
    \overline{F_+x}$. 
    If $B'\subset B$ is countable with $ \overline{B'} = B$, then clearly it suffices to construct points $x$ with $B' \subset  \overline{F_+x}$.
    Thus, from the outset we assume that $B \subset E(F,\infty)$ is countable with $FB \subset X$ bounded. 

    Recall also that in this case we have %that 
   {$%\operatorname{Lie}(Z) = 
    \fn = \R a_0$ and $Z = F$}.
    %\subset \fg$.
    For any {non-empty open}  $U \subset X$ we can choose a base point $x_0 \in U$ and apply Theorem \ref{thm: theorem on H} with {an arbitrary} $\rho >0$ %small enough so that $B_Z(\newdk{\rho}) \subset F = \{g_t\}$, 
    and with $B \subset E(F,\infty)$ countable as above. This gives us a thick set of %points  
    $h \in H$ such that, for every $y \in B$,
    \begin{equation*}
        hx_0 \in E(F_+, \infty) \cap   A\big(F_+, B_Z(\rho)y\big) = E(F_+, \infty) \cap A(F_+, y).
    \end{equation*}
    The same computations {as} in the previous proof show that % the above set of points $h x_0 \in U$
    {$E(F_+, \infty) \cap A(F_+, y)$} is stable under multiplication by $zh^-$
    %, where 
    {for arbitrary} $z \in Z$ and $h^- \in H^-$.
    % are %small
   % close enough to the identity. 
   {Thus} Lemma \ref{lem: Eilenberg} %above 
    allows us to conclude that
    \begin{equation*}
        \dim \big(U \cap E(F_+, \infty) \cap   A(F_+,y)\big) \geq \dim H + 1 + \dim H^-  = \dim X,
    \end{equation*}
    which finishes the proof.
\end{proof}
We can now focus on proving Theorem \ref{thm: theorem on H}. 
 %\vfil\eject
 \section{Equidistribution and tree-like sets on the expanding group}

In what follows we 
fix the decomposition \eqref{eq: real decomposition} of the Lie algebra of $G$ and will work with the subgroups $H,H^-,Z$ as in Definition \ref{defn: expanding, contracting and neutral groups}. A standard consequence of mixing of the $F$-action on $X$, dating back to the Ph.D.\ Thesis of Margulis \cite{M}, is %that $F$ connects the measure $m_H$ on $H$ to the measure $m_X$ on $X$
the equidistribution of $g_t$-translates of $H$-orbits as $t\to\infty$. Let us quote a special case (one-sided inequality) stated in  \cite{Kl}.

\begin{prop}[Proposition 2.4 in \cite{Kl}]\label{prop: KM equidistribution generic}
Let $V \subset H$ be a bounded measurable set. Let $Q \subset X$ be a compact set with $m_X(\partial Q) =0$, let $L \subset X$ be another compact set, and let $\eta >0$. Then there exists $t_0 := t_0(V, Q, L, \eta) >0$ such that, for all $t > t_0$ and any base point $x \in L$, we have
\begin{equation*}
    m_H\left(\{ h \in V : g_thx \in Q\}\right) > m_H(V) m_X(Q) - \eta.
\end{equation*}
\end{prop}
The above equidistirbution theorem was used in \cite{KM} to prove the existence of bounded 
$F_+$-trajectories on $X = \ggm$. We are going to use the same principle to construct trajectories visiting small neighborhoods of points of $X$. For $x\in X$ denote by $\pi_x$ the map  $G\to X$ given by $\pi_x(g) := gx$, and by  $\rho(x)$ the \textsl{injectivity radius} of $x$: $$
{\rho(x) :=}\,\sup\{r > 0: %\text{the map }G\to X,\ g\mapsto gy
\pi_x\text{ is injective on }B_G(r)\}.$$
If $K$ is a subset of $X$, let us denote by $\rho(K)$ the \textsl{injectivity radius} of $K$: $$
\rho(K) := \inf_{x\in K}\rho(x) = \sup\{r > 0: %\text{the map }%G\to X,\ 
%g\mapsto gy
\pi_x\text{ is injective on }B_G(r)\  \ \forall\,x\in K\};$$
it is known that $\rho(K) > 0$ if and only if $K$ is bounded.

\smallskip
 Let us restate the mixing assumption on the $F$–action on %$\ggm$
$X$ as follows.
\begin{prop}[Theorem 2.1.2 in \cite{KM}]\label{thm: KM mixing}
 The $F$-action on %$\ggm$
$X$ is mixing if and only if  for any two compact families of functions $\Xi, \Psi \subset L^2(m_X)$ and  for any $\eta>0$ there exists $t_0 >0$ such that, for all $t > t_0$, $\varphi \in \Xi$ and $\psi \in \Psi$,
    \begin{equation*}
        \left| \int \varphi(g_tx)\psi(x)\, dm_X(x) - \left(\int \varphi \,dm_X\right) \left(\int \psi \,dm_X\right)\right| < \eta.
    \end{equation*}
\end{prop}
%A \new{standard} consequence of \new{mixing} is %that $F$ connects the measure $m_H$ on $H$ to the measure $m_X$ on $X$
%\new{the equidistr%ibution of $g_t$-translates of $H$-orbits as $t\to\infty$}:

%But presently we need a version which allows us to vary the compact set $\ca{K}$.
%Recall that the probability measure $m_X$ on $X$ is related to a Haar measure $m_G$ on $G$ by the formula
%\begin{equation}\label{eq: integration on fibers}
%    \int fdm_G = \int_{x=g\Gamma \in X}\left( \sum_{\gamma \in \Gamma} f(g \gamma)\right) dm_X(x) 
%\end{equation}
% which holds for all  $f: G \to \R_{\geq 0}$ lower semicontinuous. (See \cite[Lemma 2.66]{Fo}.)

The next proposition is a uniform version of Proposition \ref{prop: KM equidistribution generic} with $Q$ replaced by a compact family of balls in $X$ (compact in the sense of $L^2(m_X)$).

\begin{prop}\label{prop: equidistirbution for small balls}
    Let $V \subset H$ be bounded measurable with $m_H(\partial V) =0$. Let $K \subset X$ be compact and let $0<\sigma<\rho(K)$. 
    %Assume that $r>0$ is small enough so that, for each $y \in \ca{K}$, the map
   % \begin{equation*}
   %     \pi_y: B_G(e_G,r) \to B_X(y,r);\ g \mapsto gy 
  %  \end{equation*}
  %  is a diffeomorphism.
    Let $L \subset X$ be another compact set, and let $\eta >0$. Then there exists $t_0 := t_0(V, K, \sigma, L, \eta) >0$ such that, for all $t > t_0$, base point $x \in L$ and center $y \in K$,
    \begin{equation*}
        m_H\left\lbrace h \in V: g_thx \in B_X(y,\sigma) \right\rbrace > m_H(V) m_G\big(B_G(\sigma)\big) - \eta.
    \end{equation*}
\end{prop}
\begin{proof}
    The proof is a modification of that of \cite[Proposition 2.4]{Kl2}. 
    By partitioning $V$ into finitely many smaller subsets with boundary of measure zero and noting the invariance of $m_H$, one can without loss of generality assume that $e_H\in V$, $m_H(V)<1$ and $\textup{diam}(V) < \rho(L)$. %Further, we assume $V$ is small enough so that 
  The latter condition implies that there is an open ball $U \subset G$ centered at $e_G$ and containing $V$ such that for all $x \in L$ the map
  %  \begin{equation*}
        $\pi_x$ %: U \to X;\ g \mapsto gx 
   % \end{equation*}
    is injective on $U$, and hence
    %. In particular, the formula \eqref{eq: integration on fibers} shows that, for all $x \in \ca{L}$, 
    \begin{equation*}
        m_X(Ux) = m_G(U).
    \end{equation*}
   % Without loss of generality, we can assume that $V$ is small enough so that $m_H(V) < 1$ and so that, for all $x \in \ca{L}$,\begin{equation*}\pi_x: G \to X;\ g \mapsto gx\end{equation*}is injective on a neighbourhood $U'$ of $V$.
Similarly we have, for any $0 < \alpha < 1$ and $y \in K$,
    \begin{equation*}
        m_X\big(B_X(y,\alpha \sigma)\big) = m_G\big(B_G(\alpha \sigma)\big).
    \end{equation*}
    Fix $\alpha <1$ such that
    \begin{equation}\label{eq: alpha ball close}
        m_G\big(B_G(\alpha \sigma)\big) > m_G\big(B_G(\sigma)\big) - \frac{\eta}{2}. 
    \end{equation}
    Now let $\tilde{H}$ %denote the connected subgroup of $G$ corresponding to the Lie subalgebra $\fn \oplus \fh^-$. Let 
    and $m_{\tilde{H}}$ 
    %denote a Haar measure on $\tilde{H}$. 
    be as in Definition~\ref{horosph}.
Choose a nonempty open ball $\tilde V \subset \tilde{H}$ such that $\tilde VV \subset U$ and such that for all %$z \in X$, 
    $\tilde v \in \tilde V$ and $t>0$ we have
    \begin{equation*}
        \dist_G\big(e_G,\Phi_t(\tilde v)\big) < (1-\alpha)\sigma.% \text{ and } 
    \end{equation*}
    Now for $y \in K$, $x \in L$ and $t>0$ denote
    \begin{equation*}
        W := \left\lbrace h \in V : g_t hx \in B_X(y,\sigma)\right\rbrace\ \text{ and }\ U' := \tilde V V.
    \end{equation*}
    \begin{claim}\label{claim: 99 containment}
        The set $U'x \cap g_t^{-1} B_X(y, \alpha \sigma)$ is contained in $\tilde V W x$.
    \end{claim}
    \begin{proof}%[Proof of Claim \ref{claim: 99 containment}]
        Let $\tilde v \in \tilde V$, $h \in V$ be such that $g_t\tilde vhx \in B_X(y, \alpha \sigma)$. We have
        \begin{equation*}
        \begin{split}
            \dist_X(g_t h x, y) &\leq \dist_X(g_t hx, \Phi_t(\tilde v)g_thx) + \dist_X(g_t\tilde vhx, y) \\
            &< (1-\alpha)\sigma + \alpha \sigma = \sigma.
        \end{split}
        \end{equation*}
        Thus $g_thx \in B_X(y,\sigma)$.
    \end{proof}
    We now use Proposition \ref{thm: KM mixing} with %on the compact families
    \begin{equation*}
        \Xi := \left\lbrace \mathds{1}_{U'x} %\in L^2(m_X)
        : x \in L\right\rbrace,\ 
\Psi := \{\mathds{1}_{B_X(y,\alpha \sigma}) %\in L^2(m_X)
: y \in K\},
    \end{equation*} 
    and with %error 
    $\eta$ replaced by $\eta m_{\tilde{H}}(\tilde V)/2$.
    This gives $t_0 >0$ such that, for all $t>t_0$, $x \in L$ and $y \in K$, we have
    \begin{equation*}
        \left|m_X\big(U'x \cap g_t^{-1}B_X(y,\alpha \sigma)\big) - m_G(U')m_G\big(B_G(\alpha \sigma)\big) \right| < \eta m_{\tilde{H}}(\tilde{V})/2.
    \end{equation*}
    In particular, %using the observation that after perhaps scaling 
     % the Haar measure $m_G$ is locally the product $m_{\tilde{H}} \otimes m_H$ (see \cite[Theorem 8.32]{Kn})\footnote{\comm{Please check this or add the Bourbaki reference.}} 
     using the choice of $m_{\tilde{H}}$ and Claim \ref{claim: 99 containment}, we compute
    \begin{equation*}
    \begin{split}
         m_{\tilde{H}}(\tilde{V})m_H(W) &=m_X(\tilde VWx) \\
        &\geq m_X\big(U'x \cap g_t^{-1}B_X(y,\alpha \sigma)\big) \\
        &\geq m_G(U') m_G\big(B_G(\alpha \sigma)\big) - \eta m_{\tilde{H}}(\tilde V)/2 \\
        &= m_{\tilde{H}}(\tilde V)m_H(V) m_G\big(B_G(\alpha \sigma)\big) - \eta m_{\tilde{H}}(\tilde V)/2.
    \end{split}
    \end{equation*}
    Canceling $m_{\tilde{H}}(\tilde V)$, using \eqref{eq: alpha ball close} and the fact that $m_H(V) < 1$, we get
    \begin{equation*}
        \begin{split}
            m_{H}(W) &\geq m_H(V) m_G\big(B_G(\alpha \sigma)\big) - \eta/2  \\
            &\geq m_H(V) m_G\big(B_G( \sigma)\big) - m_H(V)\eta/2 - \eta/2 \\
            &\geq m_H(V)m_G\big(B_G( \sigma)\big) - \eta
        \end{split}
    \end{equation*}
    which gives the desired result.
\end{proof}

\begin{defn}[Tessellation domains in $H$]\label{defn: tessellation}
    Following \cite{KM}, we say that a  subset $V$ of a Lie group $H$ is %said to be 
    a {\sl tessellation domain} 
    relative to a countable subset $\Lambda$ of $H$ if 
\begin{itemize}
\item[(a)] 
$m_H (\partial V) = 0$;
\item[(b)] 
$V \gamma_1 \cap V \gamma_2 = \varnothing$ for different $\gamma_1,\gamma_2 \in \Lambda$;
\item[(c)] 
$H = \bigcup\limits_{\gamma  \in \Lambda } {\overline V \gamma }.$
\end{itemize}  
For example, the set $\{(x_1,\dots, x_k): |x_i| < 1/2\}\subset \R^k$ is a tessellation domain relative to $\Z^k\subset \R^k$.
This can be generalized to tessellations of connected simply connected nilpotent Lie groups as follows.
\end{defn}
%One example of a tessellation domain is given by $V=H$ and $\Lambda = \{e_G\}$. However, the simple-connectedness of $H$ guarantees the existence of bounded tessellation domains.
%In fact, we can be more descriptive.
%We have seen from the proof of Lemma \ref{lem: contracting expanding explicit computation} that if \new{$I$ is a set} with $\{v_\lambda \in \fh_\Co: \lambda \in I\}$ is a set of eigenvectors for $\ad_{a_0}$ then the nonzero vectors in
%\begin{equation}\label{eq: tessellation domain basis}
 %   \{v_\lambda \pm \overline{v_\lambda} \in \fh : \lambda \in I\}
%\end{equation}  
%forms a basis for $\fh$. This basis leads to a tessellation domain as follows.
\begin{lem}[Proposition 3.3 in \cite{KM}]\label{lem: tessellation domain}
    %The expanding subgroup $H \subset G$ 
    Any connected simply connected nilpotent Lie group $H$ admits a bounded tessellation domain. More precisely, there exists a basis $\{b_i \}$   for $\fh = \Lie(H)$ %as in \eqref{eq: tessellation domain basis}, then
    such that for every $r>0$,  the set
    %the exponential image\footnote{\comm{We might have to change this basis to something more explicit in the future.}} of
    \begin{equation}\label{defvr}
       V_r:=\exp\Big(
       \Big\{ \sum_{i} c_i b_i \in \fh : |c_i| < r \Big\}
       \Big)
    \end{equation}
    is a tessellation domain in $H$.
\end{lem}
%\begin{proof}
%    This follows exactly from the inductive argument in \cite[Proposition 3.3]{KM} after checking that basis above gives a \textit{Malcev basis}.
%\end{proof}

Another result from \cite{KM} estimates the boundary effects  of partitioning the $\Phi_t$-image of $V$ into translates $V\gamma$:
\begin{prop}[Corollary 3.5 in \cite{KM}]\label{prop: expansion}
    %The expanding subgroup $H \subset G$ 
    For any tessellation $(V,\Lambda)$ of $H$ and any $\eta>0$ there exists $T>0$ such that
\begin{equation}
t\ge T\quad\Longrightarrow\quad\#\{\gamma\in\Lambda:
V\gamma\subset\Phi_t(V)\} 
\ge e^{\chi t}\cdot(1-\eta)\,.
\end{equation}
\end{prop}

\begin{defn}[Strongly tree-like collections and Cantor sets]\label{defn: tree like collection}\label{defn: tree-like collection}
    We say a collection of subsets $\ca{E} \subset 2^H$ is \textsl{strongly tree-like} if:
    \begin{itemize}
        \item Each $E \in \ca{E}$ is a compact set with nonempty interior.
        \item We have a partition
        \begin{equation*}
            \ca{E} = \bigcup_{k \in \Z_{\geq 0}} \ca{E}_k
        \end{equation*}
        with each $\ca{E}_k$ being finite and $\ca{E}_0$ being a singleton.
        \item If $E' \in \ca{E}_k$ with $k \in \N$, there is a unique $E \in \ca{E}_{k-1}$ with $E' \subset E$.
        \item For each $E \in \ca{E}_k$ with $k \in \Z_{\geq 0}$, there exists $E' \in \ca{E}_{k+1}$ with $E' \subset E$.
        \item If $E_1, E_2 \in \ca{E}_k$ are distinct, then $m_H(E_1 \cap E_2)=0$. 
        \item If we define 
        \begin{equation}\label{eq: diam} 
        d_k := \sup\{\on{diam}_H(E) : E \in \ca{E}_k\}, \text{ then } \lim_{k \to \infty} d_k = 0.
        \end{equation}
    \end{itemize}
    We write
    \begin{equation*}
        \cup \ca{E}_k := \bigcup_{E \in \ca{E}_k} E
    \end{equation*}
    and define the {\sl limit set} of the collection to be $E_\infty := \bigcap_{k \in \N} \cup \ca{E}_k$.
\end{defn}
Associated to a tree-like collection $\ca{E}$ we also define, for $E \in \ca{E}_k$,
\begin{equation*}
    \on{density}(\ca{E}_{k+1},E) := \frac{m_H((\cup \ca{E}_{k+1}) \cap E)}{m_H(E)}.
\end{equation*}
Further, for each $k \in \Z_{\geq 0}$, we set 
\begin{equation}\label{eq: density}
        \Delta_k := \inf\left\lbrace \operatorname{density}(\mathcal{E}_{{k+1}}, E) : E \in \mathcal{E}_k \right\rbrace
\end{equation}
\begin{thm}[Lemma 2.1 in \cite{U}]\label{thm: urbanski}
    For a tree-like collection $\ca{E}$ and the resulting %Cantor
limit set $E_\infty$, we have the dimension estimate
    \begin{equation*}
        \dim H - \dim E_\infty \leq \limsup_{k \to \infty} \frac{\sum_{j=0}^{k} \log \Delta_j}{\log d_k}.
    \end{equation*}
\end{thm}

\section{The generic and special steps of the induction}

The desired orbits in Theorem \ref{thm: theorem on H} are constructed inductively. We start with a base point $x_0$ and a tessellation domain $V \subset H$; the latter will always be chosen of the form $V = V_r$ as in \eqref{defvr}, where $r > 0$ is small enough. Then, roughly speaking, we choose a collection of  %small
subsets $\{W \subset V\}$ 
such that  $g_t$-images of $Wx_0$ 
%\new{whose} trajectories 
have a certain behavior.
%up to time $t$. 
The expansive properties of the flow ensure that each $g_{t}Wx_0 = \Phi_t(W)g_tx_0$ looks precisely like $Vx_1$ for some $x_1 \in X$, and from there we begin the process all over again.
The %requisite orbit 
two %aforementioned 
types of behavior to be considered are 1) being bounded within a certain compact set, and 2) %at certain intervals of time 
approaching a prescribed point $y$. %in the compact set.
%\smallskip 
We now formalize these two steps as follows.

\begin{prop}[The generic step]\label{prop: generic}
    Let $K\subset X$ be a compact set with $m_X(\partial {K})=0$ and let  $\eta >0 $. Then 
    one can choose $\delta > 0$ with the following property: 
    for any  $r < \delta$
   % we have the tesselation %domain 
   there exist a tessellation
    $(V_r,\Lambda_r)$ of $H$ and a time %a time 
    $t_0 >0$ such that for all points $x \in K$ {and $t\ge t_0$} we have 
    \begin{equation}\label{eq: generic conclusion}
       {\# \left\lbrace \gamma \in \Lambda_r : V_r\gamma \subset \Phi_{t}(V_r)\text{ and } V_r\gamma g_{t} x \subset K \right\rbrace \geq e^{\chi t} \big( m_X(K) - \eta\big)},
    \end{equation}
   where $\chi$ is as in \eqref{eq: chi t}.
\end{prop}
\begin{proof}
    This is a simplified version of \cite[Proposition 3.6]{KM}. We repeat the argument here for the sake of %the reader
clarity.
    The proposition is true if $m_X(K)=0$ by choosing $\delta$ and $t_0$ arbitrarily. {Thus we can assume that the interior of $K$ has positive measure.}
   \smallskip
   
%5Let 
{Given $\eta>0$},
%be small enough so that
 %   \begin{equation*}
%        \new{\eta} < m_X(K)  - m_X(K)^2.
%    \end{equation*}
%    We 
we choose a compact subset $K' \subset K$  and 
$\delta > 0$
so that:
    \begin{enumerate}
        \item[(a)] %$K'\subset K$, 
    $m_X(\partial K') =0$ and $m_X(K') > m_X(K) - \eta/3$;
        \item[(b)] $V_\delta V_\delta^{-1} K' \subset K$ (hence $V_r V_r^{-1} K' \subset K$ for any %$V = V_r$ with 
        $0< r < \delta$).
        %;
       % \item[(c)] $\text{diam}_G(V) < \delta$.
    \end{enumerate}
    For any $r$ as above let $V = V_r$ and choose $\Lambda = \Lambda_r$ such that 
    %{in view of  
    $(V,\Lambda)$ is a tessellation of $H$. Then for all $x \in X$ and $t>0$ we have the following estimate:
\begin{equation}\label{eq: bdd page 19}
\begin{split}
 &\# \left\lbrace \gamma \in \Lambda: V\gamma \subset \Phi_{t}(V)\text{ and } V\gamma g_{t}x \subset K \right\rbrace 
    \\\geq\ \   &\#\left\lbrace \gamma \in \Lambda : V\gamma \subset \Phi_{t}(V) \text{ and }  V\gamma g_{t} x \subset VV^{-1}K' \right\rbrace \\
    \geq\ \  &\#\left\lbrace \gamma \in \Lambda : V\gamma \subset \Phi_{t}(V) \text{ and }  \gamma g_{t} x \in V^{-1}{K}' \right\rbrace \\
    \geq\ \   &\#\left\lbrace \gamma \in \Lambda : V\gamma \subset \Phi_{t}(V) \text{ and } V\gamma g_{t} x \cap K' \neq \varnothing \right\rbrace \\
    =\ \   &\# \left\lbrace \gamma \in \Lambda : V\gamma \subset \Phi_{t}(V) \right\rbrace - \#\left\lbrace \gamma \in \Lambda : V\gamma \subset \Phi_{t}(V)\text{ and }   V\gamma g_{t} x \cap K' = \varnothing\right\rbrace
    \\
    \geq\ \   &\# \left\lbrace \gamma \in \Lambda : V\gamma \subset \Phi_{t}(V) \right\rbrace - \frac{ m_H\big(\Phi_{t}(V) \smallsetminus \{ h \in \Phi_{t}(V): hg_{t}x \in K'\}\big)} {m_H(V)}\\
    \geq\ \   &\# \left\lbrace \gamma \in \Lambda : V\gamma \subset \Phi_{t}(V) \right\rbrace - \frac{ m_H\big(\Phi_{t}(V)\big) - m _H(\{h \in \Phi_{t}(V): hg_{t}x \in K'\})}{m_H(V)}\\
    =\ \   &\# \left\lbrace \gamma \in \Lambda : V\gamma \subset \Phi_{t}(V) \right\rbrace -e^{\chi t}  \left(1-  \frac{ m _H(\{h \in V: g_{t}hx \in K'\})}{m_H(V)}\right)
\end{split}
\end{equation}
(the last step follows from \eqref{eq: chi t} and Definition \ref{defn: congujation}).
   {Now} apply Proposition \ref{prop: KM equidistribution generic} with 
    %parameters
   % \begin{equation*}
    %    V,\ \text{compact set with null boundary } K', \text{ error as }\ \frac{\new{\eta} m_H(V)}{3},\ \text{ and basepoint allowed to vary in } K.
    %\end{equation*}
   {$Q$ replaced by $K'$, $L$ replaced by $K$, and $\eta$ replaced by $\frac{\eta m_H(V)}{3}$}.
    This gives us a time $t_1 >0$ such that, for all  $x \in K$ and $t \geq t_1$,
   \begin{equation}\label{eq: equidistribution generic cor}
          {m_H \big(\left\lbrace h \in V: g_thx \in  {K'} \right\rbrace\big)  > m_H(V)\big(m_X( {K'}) - \eta /{3}\big) %> m_H(V)\big(m_X( {K}) -\new{\eta}  \big)
          }.
    \end{equation}
   
%Thus, we obtain exactly as in \cite[Proposition 3.4 and 
On the other hand, %\cite[Corollary 3.5]{KM} 
Proposition \ref{prop: expansion} yields a time $t_2 > 0$ such that, for all $t \geq t_2$,
\begin{equation}\label{eq: minuend}
   \#\left\lbrace \gamma \in \Lambda : V\gamma \subset \Phi_t(V)\right\rbrace \geq %\frac{m_H\big(\Phi_t(V)\big)}{m_H(V)} \left(1- \frac{\new{\eta}}{3} \right)  =
   e^{\chi t}  \left(1 - {\eta}/{3} \right).
\end{equation}
Letting  $t_0 := \max\{t_1, t_2\}$ and combining \eqref{eq: bdd page 19}, \eqref{eq: equidistribution generic cor} and \eqref{eq: minuend}, for any $t > t_0$ we obtain
\begin{equation*}
\begin{split}
    \#\left\lbrace \gamma \in \Lambda : V\gamma \subset \Phi_{t}(V)\text{ and } V\gamma g_{t}x \subset K \right\rbrace &\geq e^{\chi t}   (1 - \eta/3) - e^{\chi t}  \big( 1  - m_X (K') + \eta/3 \big) \\
    &= e^{\chi t}   \big( m_X (K')  - 2\eta/3 \big) \geq e^{\chi t}  \left( m_X(K) - \eta\right),
    %\\
    %&\geq \chi_{t_0}  \left( m_X (K)\right)^2
\end{split}
\end{equation*}
which is the desired conclusion.
\end{proof}

We now come to the special step of the %induction
procedure which ensures that our constructed orbits have prescribed limit points.

\begin{prop}[The approach step]\label{prop: approach}  For any compact $K\subset X$  there exists $\delta > 0$ %with the following properties. F
such that for any $r,\sigma < \delta$
   one has the tessellation %domain 
   $(V_r,\Lambda_r)$ and a time $t_0 >0$ such that the following %property 
   holds. For any $t \geq t_0$ and $\varepsilon>0$ %, there exists a time 
   one can find $s_\varepsilon>s'_\varepsilon>t$ such that for  any $x\in K$ and any $y\in K$ with  \eq{conditionsony}{Fy\subset K\text{ and }\dist_X(Fy, \partial K)>1}  there exists $\gamma = \gamma(x,y,\varepsilon)\in \Lambda_r$ %such that the following conditions are 
   satisfying
   \begin{itemize}
   \item[\rm (i)] The set $W:= \Phi_{-s_\varepsilon}(V_r \gamma)$ is a subset of $V_r$.
       \item[\rm (ii)]  For all $s\in [t,s_{\varepsilon}]$ we have $g_{s}Wx\subset  K$.
        \item[\rm (iii)] %There exists $s\in[t  , s_\varepsilon]$ such that 
        For all $h \in W$, $\dist_X\big(g_{s'_\varepsilon}hx,B_Z(\sigma)y\big) < \varepsilon$. 
   \end{itemize}
\end{prop}
\begin{proof}[Proof of Proposition \ref{prop: approach}]
    %There is a 
    Take $\delta_1>0$ such that the multiplication map from $H\times H^- \times Z$ to $G$ maps onto $B_G(\delta_1)$ and has an inverse on $B_G(\delta_1)$. That is, we have a smooth and bi-Lipschitz (with respect to the supremum metric of the product space) map
    \begin{equation*}
        B_G(\delta_1) \to H \times H^{-} \times Z
    \end{equation*}
    which composes with multiplication to give the identity map on $B_G(\delta_1)$.
    In particular, there exists a constant $c \in (0,1)$ such that, for all $r \leq \delta_1$, 
    \begin{equation*}
        g \in B_G(cr) \implies g = hh^-z \text{ for some } (h,h^-,z) \in B_H(r)  \times B_{H^-}(r)\times B_Z(r).
    \end{equation*}

    We may shrink $\delta_1$ further to ensure that
    \begin{itemize}
        \item $\delta_1 < 1/3$;
        \item $\operatorname{diam}_X(V_{\delta_1}) < 1/3$;
        \item $\Phi_t\big(B_{H-}(\delta_1)\big) \subset B_G(1/3) \text{ for all } t \geq 0$;
        \item $\Phi_t\big(V_{\delta_1}(\overline{V_{\delta_1}})^{-1}\big) \subset B_G(1/3)$ for all $t \leq 0$;
        \item $\Phi_t\big(B_{Z}(\delta_1)\big) \subset B_G(1/3)$ for all $t \in \R$.
    \end{itemize}
    Now fix $\delta < \min\{\rho(K), \delta_1\}$.
    Thus, for every point $x \in K$ and $r < \delta$ the map
    \begin{equation*}\label{eq: product decomposition}
        B_H(r)  \times B_{H^-}(r) \times B_Z(r) \to X; \ (h,h^-,z) \mapsto hh^-z x
    \end{equation*}
    maps onto a neighborhood of $B_X(x,cr)$.
        
    Fix $r, \sigma < \delta$ and consider the tessellation domain $V = V_r$ from Lemma \ref{lem: tessellation domain} with a corresponding discrete set $\Lambda = \Lambda_r$.
    Apply Proposition \ref{prop: equidistirbution for small balls} with parameters $V$, $K$, $\sigma$, $L$ and $\varepsilon$ there replaced by $V_{r/2}$, $K$, $c\sigma$, $K$ and  
    \begin{equation*}
        \frac{m_H(V_{r/2})m_G(B(c\sigma))}{2}
    \end{equation*} 
    respectively.
    From this we obtain a $t_0>0$ such that, for all $t > t_0$, $x \in K$, $y \in K$, 
    \begin{equation}\label{eq: equidistirbution in small balls}
        m_H\left(\left\lbrace h \in V_{r/2}: g_t hx \in B_X(y,c\sigma) \right\rbrace\right) > \frac{m_H(V_{r/2})m_G(B(c\sigma))}{2}.
    \end{equation}
    We further assume $t_0$ is large enough so that
    \begin{equation}\label{eq: t0 large}
        \Phi_{-t}\left((B_H(\sigma))^{-1}\right) V_{r/2} \subset V_{3r/4} \text{ for all } t \geq t_0.  
    \end{equation}

    We now fix $t \geq t_0$ and $\varepsilon > 0$, 
    choose $s_1 > 0$ large enough so that
    \begin{equation}\label{eq: s1 large}
        \Phi_{s_1}\big(B_{H^-}(\sigma)\big) \subset B_{H^-}(\varepsilon/2),
    \end{equation}
    and then choose $s_\varepsilon > s'_\varepsilon:=s_1 +t$ so that
    \begin{equation}\label{eq: sepsilon large 1}
        \Phi_{-s_\varepsilon}\big(V(\overline{V})^{-1}\big)V_{3r/4} \subset V 
    \end{equation}
    and 
    \begin{equation}\label{eq: sepsilon large 2}
    \Phi_{s'_\varepsilon - {s_\varepsilon}}\big(V(\overline{V})^{-1}\big) \subset B_{H}(\varepsilon/2).
    \end{equation}

   Now fix $x\in K$ and  $y\in K$ such that \equ{conditionsony} holds. Inequality \eqref{eq: equidistirbution in small balls} applied with $y$ replaced by $g_{-s_1}y$, which is also contained in $K$, gives us an $h_1 \in V_{r/2}$ with 
    \begin{equation*}
        g_{t}h_1 x \in B_X(g_{-s_1}y, c\sigma).
    \end{equation*}
    The assumption on $\delta$ allows us to write $g_th_1 x$ as a product.
    \begin{equation}\label{eq: written as product}
        g_{t} h_1 x = h_2h^-zg_{-s_1}y \text{ for } h_2 \in B_H(\sigma), \ h^- \in B_{H^-}(\sigma), \ z \in B_Z(\sigma).
    \end{equation}
    In particular, we have
    \begin{equation}\label{eq: gt h perturbed}
        g_{t} \Phi_{-t}(h_2^{-1})h_1x = h^- zg_{-s_1}y.
    \end{equation}
    The stipulation in \eqref{eq: t0 large} ensures that we can write
    \begin{equation}\label{eq: h perturbed still in V}
        \Phi_{-t}(h_2^{-1})h_1x = h_3 x \text{ where } h_3 \in V_{3r/4}.
    \end{equation}
    Combining \eqref{eq: gt h perturbed} and \eqref{eq: h perturbed still in V}, we have
    \begin{equation}\label{eq: h3 hits the spot}
        g_{t} h_3x = h^-zg_{-s_1}y.
    \end{equation}
    Choose and fix $\gamma \in\Lambda$ with $\Phi_{s_\varepsilon}(h_3) \in \overline{V}\gamma$. 
    We write 
    \begin{equation}\label{eq: h3 as v3}
        h_3 = \Phi_{-s_\varepsilon}(h_4 \gamma) \text{ for some } h_4 \in \overline{V},
    \end{equation}  which gives
    \begin{equation*}
    %\begin{split}
        W %&
        := \Phi_{-s_\varepsilon}(V\gamma)   %\\
        %&
        = \Phi_{-s_\varepsilon}(Vh_4^{-1})h_3.
    %\end{split}
    \end{equation*}
    Since $h_3 \in V_{3/4}$ and since we have the containment \eqref{eq: sepsilon large 1}, we see that 
    %\begin{equation*}
        $W \subset V$.
   % \end{equation*}
    This proves Assertion (i) of the Proposition.
    We now compute, for $t \leq s \leq s_\varepsilon$,
\begin{equation}\label{eq: part ii computation}
        \begin{split}
            g_s Wx &=g_s \Phi_{-s_{\varepsilon}}(V\gamma)x \\
        &\underset{\eqref{eq: h3 as v3}}\subset g_s\Phi_{-s_\varepsilon}\big(V(\overline{V})^{-1}\big)h_3x  =   \Phi_{s-s_\varepsilon}\big(V(\overline{V})^{-1}\big)g_{s}h_3x \\
            &\underset{\eqref{eq: h3 hits the spot}}= \Phi_{s-s_\varepsilon}\big(V(\overline{V})^{-1}\big)g_{s-t} {h^-}zg_{-s_1}y \\
            &= \Phi_{s-s_\varepsilon}\big(V(\overline{V})^{-1}\big)\Phi_{s-t}({h^-}) \Phi_{s-t} (z)g_{s-t-s_1}y \\
            &\underset{\eqref{eq: written as product}}\subset \Phi_{s-s_\varepsilon}\big(V(\overline{V})^{-1}\big) \Phi_{s-t}\big(B_{H^-}(\sigma)\big) \Phi_{s-t}\left(B_Z(\sigma)\right)Fy\\
            &\subset B_G(1/3 + 1/3 + 1/3)Fy \subset K,
        \end{split}
    \end{equation}
    %In the above equation, we have used \eqref{eq: h3 as v3} for the first containment, \eqref{eq: h3 hits the spot} for the second equality, the condition \eqref{eq: written as product} on $\new{h^-}$ for the third containment, and 
    where we have used the stipulations on $\delta$ and \equ{conditionsony}  for the last two containments.
    This proves Assertion (ii).
    %of the Proposition.

    Finally we compute, as in \eqref{eq: part ii computation},
    \begin{equation*}
        \begin{split}
    g_{{s'_\varepsilon}}Wx 
            \underset{\eqref{eq: h3 as v3}}\subset & \Phi_{s'_\varepsilon-s_\varepsilon}\big(V(\overline{V})^{-1}\big)g_{s_1+t}h_3x \underset{\eqref{eq: h3 hits the spot}}= \Phi_{{s'_\varepsilon} - s_\varepsilon}\big(V(\overline{V})^{-1}\big)\Phi_{s_1}(h^-)zy \\
            \subset \ \  & B_H(\varepsilon/2)B_{H^-}(\varepsilon/2)B_Z(\sigma)y,
        \end{split}
    \end{equation*}
    where \eqref{eq: sepsilon large 2} and \eqref{eq: s1 large} are used to justify the last containment. This proves Assertion (iii) and finishes the proof of the proposition.
\end{proof}

\ignore{\begin{proof}[Sketch of proof]
\ 

 \begin{itemize}
       \item
Take $\delta$ to be the minimum of the injectivity radius of $K$ (or maybe $1/4$ of the injectivity radius of  the $1$-neighborhood of $K$ to be safe)  and $\varepsilon_0$ which you defined in Remark \ref{rem: e0 assumption}. 
 \item Then given small $r$ and $\sigma$ define $t_0'$ using Proposition \ref{prop: equidistirbution for small balls} with $V = V_{r/2}$, $L = K$, and $\sigma$ in place of $\varepsilon$. 
  \item Now we take $t' > t_0'$, $x,y$ and $\varepsilon$. 
Take  $s_1$ such that \eq{defs1}{\Phi_{s_1}\big(B_{H^-}(\sigma)\big)\subset B_{H^-}(\varepsilon/2),} and let $z = g_{-s_1}y$.
Proposition \ref{prop: equidistirbution for small balls} guarantees a point $h'\in V_{r/2}$ such that $$g_{t'}h'x\in B_X(z,\sigma/2).$$
\comm{A word of caution: I am using $2$ and $1/2$ in place of constants $\new{C},C_2$ you used in  Remark \ref{rem: e0 assumption}. This has to be checked carefully. Also we have the estimate    $\new{C}r \le \textup{diam}(V_r) \le c_2r$.}
\item Now use the local product structure to find $h_1\in B_H(\sigma/2)$ such that $$h_1g_{t'}h'x\in B_{\tilde H}(\sigma/2)z.$$
Writing $h_1g_{t'}h' = g_{t'}\Phi_{-t'}(h_1) h'$, we see that it makes sense to work with $$h_0 := \Phi_{-t'}(h_1) h'$$ which should be in $V$; we conclude that $g_{t'}h_0x\in B_{\tilde H}(\sigma/2)z$, and therefore for any $s > 0$ one has 
\eq{bdd}{g_{t'+s}h_0x\in \Phi_s\big(B_{\tilde H}(\sigma/2)\big)g_sz,} which lies in a small neighborhood of $K$ because $Fz\subset K$.
\item Then use the  local product structure again to decompose $B_{\tilde H}(\sigma/2)$ and get 
$$g_{s_1}g_{t'}h_0x\in \Phi_{s_1}\big(B_{\tilde H}(\sigma/2)\big)g_{s_1}z \subset B_{H^-}(\varepsilon/2)B_Z(\sigma) y ,$$
where the last inclusion is due to \equ{defs1}.
\item Finally choose $s_2$ such that 
\eq{defs2}{\Phi_{-s_2}\big(B_{H}(2r)\big)\subset B_{H}(\varepsilon/2),}
then pick $\gamma\in\Lambda$ such that
$g_{s_2}g_{s_1}g_{t'}h_0\in\overline{V}\gamma$, and define $W$ by \equ{defw}. 
Then for $h\in W$ and all $t\le t' + s_1 + s_2$ the point $g_thx$ will be pretty close to $g_{t}h_0x$.
\item It remains to check that (i) follows from \equ{defs2}  and (ii) from \equ{bdd}.
\end{itemize}\end{proof}}

\section{%Marked points of the tesselation and the Cantor set
Proof of Theorem \ref{thm: theorem on H}}

In this section our goal is to prove Theorem \ref{thm: theorem on H}, that is, for given $x_0\in X$, construct many points of the form $hx_0$ such that their $F_+$-orbits are bounded, and, in addition, a given countable subset of $X$ is  in their orbit closure. Our proof is a modification of a well-known scheme of constructing points with bounded orbits due to the second-named author and Margulis \cite{KM}, see also \cite{Kl2, Kl}. Before dealing with the general case it will be instructive   to briefly describe this construction in a special case.

\begin{proof}[Proof of Theorem \ref{thm: theorem on H} in the special case $B = \varnothing$]
Take  $x_0\in X$ and a nonempty  open subset
$U$ of $H$. We need to prove that the \hd\ of the set \eq{bdd}{
\{h \in U:F_+ hx_0 \text{ is bounded}\}} is equal to $\dim H$. %Using
%Corollary 2.5, we 
%can find a subset $V$ of $U$ such that $Vx$ is disjoint from
%$Z$. Then, r
Replacing $x_0$ by $hx_0$ for some $h\in U$ we can assume that
 $U$ is a neighborhood of identity in $H$.
 
 Pick a compact set
$K\subset X$ with $m_X(\partial K) = 0$ containing $x_0$, and let \eq{defeta}{\eta := m_X(K) - m_X(K)^2.}
%fix an arbitrary $0 < \eta < m_X(K)$.
Then take $\delta>0$ furnished by Proposition \ref{prop: generic} and choose $0 < r < \delta$ such that $V_r\subset U$. By shrinking $r$ further we can assume that \eq{diamV}{ \on{diam}_H(V_r) < 1/2.}
As before we denote $V = V_r$ and let $ \Lambda =  \Lambda_r$ be a corresponding discrete subset of $H$. The proposition gives us  the value $t_0 >0$  such that \eqref{eq: generic conclusion} holds for  all $x \in K$ {and $t\ge t_0$}. Fix  $t > t_0$ and let $\Lambda(x)$ denote the set appearing in the left hand side of \eqref{eq: generic conclusion}, i.e.
\eq{lambdax}{\Lambda(x):=\left\lbrace \gamma \in \Lambda : V\gamma \subset \Phi_{t}(V)\text{ and } V\gamma g_{t} x \subset K \right\rbrace.}

To estimate the \hd\ of the set  \equ{bdd} from below, for each $x\in K$ we are going to construct a strongly
tree-like   collection $\ca{E}$  of subsets of $H$ such that its limit set $$\ca{E}_\infty:= \bigcap_{k \in \N} \cup \ca{E}_k$$ is contained in the set  \equ{bdd}.   We first let $\ca{E}_0 := \{\overline{V}\}$. %for all $x$. 
%Note that $\overline{V_r}x\subset K$ in view of \equ{diamV} and 
Then define 
$$
\ca{E}_{1} :=
\left\{\Phi_{-t}(\overline{V}\gamma): \gamma \in
\Lambda(x_0)\right\}.%\tag 3.5
$$
That is, elements of $
\ca{E}_{1} $ are of the form $E = \Phi_{-t}(\overline{V})h$ for some $h\in H$, chosen so that $E\subset \overline{V}$ and $g_tEx_0 \subset K$.
Similarly, elements of $
\ca{E}_{k} $ will be of the form $$E = \Phi_{-kt}(\overline{V})h\subset \overline{V}\text{  for some }h\in H,$$ and it will follow from our inductive construction that \eq{inK}{E\in  \ca{E}_k\Longrightarrow g_{it}Ex_0 \subset K \text{ for }i = 1,\dots,k.}
Namely, 
if $\ca{E}_{k}$ is chosen 
%for all $x\in K$ 
and we take $E\in\ca{E}_{k}$ such that $$E = \Phi_{-kt}(\overline{V})h \subset \overline{V}\text{ and }g_{it}Ex_0 \subset K \text{ for }i = 1,\dots,k,$$  define the subcollection
 \eq{induction}{\ca{E}_{k+1}(E):= \left\{ \Phi_{-(k+1)t}(\overline{V}\gamma) %g_{kt}hx
 :\gamma \in
\Lambda({g_{kt}hx_0})\right \},}
 and let $ \ca{E}_{k+1}:= \cup_{E\in \ca{E}_{k}} \ca{E}_{k+1}(E)$. Then $
 \ca{E}_{k+1}(E)$ consists precisely of elements of $ \ca{E}_{k+1}$ contained in $E$ (i.e.\ "children" of $E$), and it follows from the definition of $\Lambda({g_{kt}hx_0})$  that $g_{(k+1)t}Ex_0 \subset K$ for any $E\in  \ca{E}_{k+1}$.
\smallskip

The first five parts of Definition \ref{defn: tree-like collection}  follow readily from the construction and
from $(V,\Lambda)$ being a tessellation of $H$. 
Also, if we  define
\begin{equation*}\label{lambdamindef}
    \lambda_{\text{min}} := \min\{\Re \lambda: \lambda \text{ is an eigenvalue for } \ad_{a_0}\otimes 1 \text{ with } \Re \lambda >0\},
\end{equation*}
then a computation using the exponential map shows that for all $t$ sufficiently large, 
\begin{equation}\label{eq: lambda min}
    \dist_H\big(\Phi_{-t}(h),e_G\big) \leq e^{-\frac{\lambda_{\text{min}}t}{2}}\dist_H(h,e_G) \text{ for all } h \in V.
\end{equation}
Note that the factor $1/2$ in the exponent is used to get rid of the constants that occur when comparing distances on $H$ with norms in $\fg$. By increasing $t$ we can assume that \eqref{eq: lambda min} is satisfied.
%Let $E \in \ca{E}_m$ for some $m \in \N$. 
From   \equ{induction} and the   right-invariance of the metric  we can see that  for any $x\in K$ the diameter of any $E\in \ca{E}_{k}$ is equal to 
$\on{diam}_H\big(\Phi_{-kt}(V)\big)$. Thus it follows from \equ{diamV} and \eqref{eq: lambda min}   that
\begin{equation}\label{eq: dm estimate generic}
    d_k =  \sup\{\on{diam}_H(E) : E \in \ca{E}_k\} \leq e^{-\frac{\lambda_{\on{min}}kt}{2}}.
\end{equation} 
Hence \eqref{eq: diam} holds, and $E$ is a strongly tree-like collection.   Clearly \equ{inK} implies that  $g_{kt} E_\infty x_0\subset K$ for all $k\in\N$, hence 
%Further, to show that %$hx\in E(F^+,\infty)$ for any 
%$$\new{E_\infty}:= \bigcap_{k \in \N} \cup \ca{E}_k$$
%is contained in the set  \equ{bdd}, one can easily see by induction using the definition of sets $\Lambda_r(x,K,t)$ that $g_{it} hx\in K$ for  $h\in\cup\ca{E}_k$ and $i = 0,\dots,k$; hence $g_{kt} \new{E_\infty}x\subset K$ for all $i = 0,1,\dots$, which implies that 
$F_+E_\infty x_0$ is contained in  $g_{[-t,0]} K$.

It remains to estimate the \hd\ of $E_\infty$ from below. 
Using  \equ{induction} %and Proposition \ref{prop: generic}  
one can easily see that for any $k= 0,1,\dots$ %, $x\in K$ 
and 
for $E \in \ca{E}_k$ of the form $E = \Phi_{-kt}(\overline{V})h$, 
\begin{equation*}
    \on{density}(\ca{E}_{k+1},E) = \frac{m_H\big(\cup\ca{E}_{k+1}(E)\big)}{m_H(E)} \ge e^{-\chi t} \#\Lambda({g_{kt}hx_0})%(,K,t)
    .\end{equation*}
Thus Proposition \ref{prop: generic} and \equ{defeta} imply that for any $k= 0,1,\dots$  one has
\begin{equation}\label{eq: densitybound}
        \Delta_k = \inf\left\lbrace \operatorname{density}(\mathcal{E}_{{k+1}}, E) : E \in \mathcal{E}_k \right\rbrace \ge m_X(K) - \eta = m_X(K)^2,
\end{equation}
and an application of Theorem \ref{thm: urbanski} %for any $x\in K$ 
yields
\begin{equation*}
\begin{aligned}
        \dim H - \dim  {E}_\infty \leq &\limsup_{k \to \infty} \frac{\sum_{j=0}^{k} \log \Delta_j}{\log d_k}\\
        \underset{\text{\rm by } \eqref{eq: dm estimate generic} \text{\rm \ and }\eqref{eq: densitybound}} \leq &\limsup_{k \to \infty} \frac{(k+1)\log\big(\frac1{m_X(K)^2}\big)}{\frac12{\lambda_{\on{min}}kt}} = \frac{4\log\big(\frac1{m_X(K) }\big)}{{\lambda_{\on{min}}t}}.
       \end{aligned}
    \end{equation*}
The right hand side of the above inequality can be made arbitrary small by increasing $t$, which implies that  the \hd\ of the set  \equ{bdd} is equal to $ \dim H$. 
 \end{proof}
%We now have
Now let us show how the above scheme can be modified to produce a proof of the general case of Theorem \ref{thm: theorem on H}. The idea is 
to perform lots of generic steps as before, but occasionally insert approach steps.
\begin{proof}[Proof of Theorem \ref{thm: theorem on H} when $B\neq \varnothing$] %In this case 
Fix a surjection 
%\begin{equation*}
    $\beta: \N \to B$
%\end{equation*}
such that  $\beta^{-1}(y)$ is infinite for every $y 
\in B$,
%We also fix 
and a sequence of positive real numbers $(\varepsilon_n)_{n \in \N}$ converging to zero.

Take a base-point $x_0 \in X$. As before, it suffices to choose a neighborhood $U$ of the identity in $H$ and show that the set of $h \in U$ such that $hx_0$ lies in
\begin{equation}\label{eq: da one and only}
    E(F_+, \infty) \cap \bigcap_{y \in B} A(F_+, B_Z(\rho)y)
\end{equation}
has full dimension in $H$.
Choose a compact set $K$ with
\begin{equation*}
    x_0 \in \text{int}(K),\ m_X(\partial K)=0,\ FB \subset K\text{ and } \dist_X(FB, \partial K)>1.
\end{equation*}
Apply Proposition \ref{prop: generic} with this $K$ and $\eta = m_X(K) - m_X(K)^2$.
This gives us a $\delta_1 >0$. Now we may further apply Proposition \ref{prop: approach} to $K$ to  obtain another $\delta_2 >0$.
We fix
\begin{equation*}
   \sigma < \rho/2 \ \text{ and } \ r < \min\{\delta_1, \delta_2, \dist_X(x_0, \partial K)\} \text{  such that \equ{diamV} holds},
\end{equation*}
and obtain a tessellation domain $(V_r, \Lambda_r)$, which we write as $(V,\Lambda)$, and  $t_0>0$ such that the conclusions of Propositions \ref{prop: generic} and \ref{prop: approach} hold.
Fix $ t>t_0$, and, as in the proof of the special case, by increasing $t$ if needed  assume that \eqref{eq: lambda min} is satisfied. 
This way:
\begin{itemize}
    \item for any $x\in  K$ we can define $\Lambda(x)\subset \Lambda$  as in \equ{lambdax} with $\# \Lambda(x) \geq e^{\chi t}  m_X(K)^2$;
    \item for any $\varepsilon>0$  one can choose    $s_\varepsilon>s'_\varepsilon>t$ such that   any $x\in K$ and  $y\in B$   come with $\gamma = \gamma(x,y,\varepsilon)\in \Lambda$ satisfying conditions (i)--(iii) of Propositions \ref{prop: approach}. \end{itemize} 
We now formalize the process of alternating generic and approach steps.
%We get a $t_0>0$ with respect to which both Propositions hold and, with respect to a fixed $ t>\new{\max\{t_0,1\}}$, we further apply Proposition \ref{prop: approach} a countable number of times, once for each $\varepsilon_n$, to obtain a sequence of times 
%\begin{equation*}
%    s_n := s_{\varepsilon_n} \text{ for each } n \in \N.
%\end{equation*}
%Both Propositions give a subset of $\Lambda$ with respect to certain basepoints within $K$ and additionally, in the case of Proposition \ref{prop: approach}, with respect to a certain parameter $\varepsilon$. We formalize the output as follows. 
Namely  fix   an injection $f : \N \to \N$  to be determined later; our strategy will consist of performing approach steps whenever $k$ is of the form $f(n)$, and generic steps otherwise. Specifically, for each $x \in K$  and $k \in \N$ we define a subset $\Lambda(x,k)$ of $\Lambda$ by
\begin{equation}\label{deflambdaxk}
    \Lambda(x,k):= \begin{cases} 
    %\left\lbrace \gamma \in \Lambda: V\gamma \subset \Phi_{t}(V) \text{ and } V\gamma g_{t}x \subset K \right\rbrace 
    \Lambda(x) \text{ as in \equ{lambdax} (from applying Proposition \ref{prop: generic})} & \text{ if }\ k \notin f(\N);\\
    \left\lbrace \gamma\big( x,\beta(n),\varepsilon_n\big) \right\rbrace \text{ (from applying Proposition \ref{prop: approach})}  & \text{ if }\ k = f(n).
     \end{cases}
\end{equation}
We also define a sequence of times $(t_k)_{k \in \N}$ by
\begin{equation*}
    t_k := \begin{cases}
        t & \text{if } k \notin f(\N), \\
        s_{\varepsilon_n} & \text{if } k = f(n),
    \end{cases}
\end{equation*}
and set $T_0 = 0$ and $ T_k := \sum_{j=1}^k t_j \text{ for each } k \in \N$, or,  equivalently,
\begin{equation}\label{tk}
    T_k = \big(k-\#f^{-1}(\{1,\dots,k\})\big) t + \sum_{f(n)\le k} s_{\varepsilon_n} = k  t + \sum_{f(n)\le k} (s_{\varepsilon_n}-t).
\end{equation}
%where $n_k := \#f^{-1}(\{1,\dots,k\})$ is the number of special steps up to step $k$.}

We now come to the inductive construction. Define
\begin{equation*}
    \ca{E}_0 := \{\overline{V}\},\ p_0: \ca{E}_0 \to H;\ \overline{V} \mapsto e_H.
\end{equation*} 
Assume we have defined $\ca{E}_j$ and $p_j: \ca{E}_j \to H$ for $j=0,1,\dots, k$ with the properties that
\begin{enumerate}
    \item[(a)] Each $E \in \ca{E}_j$ is of the form
    \begin{equation*}
        \Phi_{-T_j}(\overline{V})p_j(E)\ \text{ with }\ g_{T_j}p_j(E) x_0 \in K.
    \end{equation*}
    \item[(b)] If $E \in \ca{E}_j$, then 
    \begin{equation*}
        g_{[0,T_j]}Ex_0 \subset g_{[-t,t]}K.
    \end{equation*}
    \item[(c)] %If there is a 
    For any $j \in \{1,\dots, k\}$ with $j = f(n)$ for some $n \in \N$   there exists   $T \in [T_{j-1}, T_j]$ such that, for all $E \in \ca{E}_j$,
    \begin{equation*}
        g_{T}Ex_0 \subset B_Z(\sigma) B_X\big(\beta(n), \varepsilon_n\big).
    \end{equation*}
\end{enumerate}
We then proceed to define, for each $E \in \ca{E}_k$, the collection of children $\ca{E}_{k+1}(E)$ as %$\new{E'}$ of the form
\begin{equation}
\begin{aligned}
\label{eq: children inductive}
 \ca{E}_{k+1}(E) 
 := &\text{ the set of all } E' =  \Phi_{-T_{k+1}}(\overline{V}\gamma)p_k(E), \\ %\left(= \Phi_{-T_{k+1}}(\overline{V})\Phi_{-T_{k+1}}(\gamma)p_k(E)\right), 
 &\text{ where } 
 \gamma \in \Lambda(g_{T_k}p_k(E)x_0, k+1).
    \end{aligned}
\end{equation}
We further define, for each $E'$ as above,
\begin{equation*}
    p_{k+1}(E') = \Phi_{-T_{k+1}}(\gamma)p_k(E).
\end{equation*}
Setting
%\begin{equation*}
    $\ca{E}_{k+1}:= \bigcup_{E \in \ca{E}_k}\ca{E}_{k}(E)$,
%\end{equation*} 
we now proceed to verify the conditions (a)--(c) of the induction. For (a), note that each $E' \in \ca{E}_{k+1}(E)$ is of the form $\Phi_{-T_{k+1}}(\overline{V}) p_{k+1}(E')$, and we must show that $g_{T_{k+1}}p_{k+1}(E')x_0 \in K$. We compute
\begin{equation*}
    %\begin{split}
        g_{T_{k+1}}p_{k+1}(E')x_0 %&
        = g_{T_{k+1}}\Phi_{-T_{k+1}}(\gamma)p_k(E)x_0 %\\
        %&
        = \left(\gamma g_{t_{k+1}}\right)g_{T_k}p_k(E)x_0,
   % \end{split}
\end{equation*}
%And this belongs to 
which is contained in $K$ either from the definition of the set in \eqref{eq: generic conclusion}, or from  Proposition \ref{prop: approach}(ii). (Note that $V$ contains the identity.)

We come to (b). First note that each $E' \in \ca{E}_{k+1}(E)$ as in \eqref{eq: children inductive} is a subset of $E$:
\begin{equation*}
    E' = \Phi_{-T_{k+1}}(\overline{V}\gamma)p_k(E) = \Phi_{-T_k}\circ\Phi_{-t_{k+1}}(\overline{V}\gamma)p_k(E) \subset \Phi_{-T_k}(V)p_k(E) = E
\end{equation*}
where the containment holds from the either Proposition \ref{prop: generic} or  Proposition \ref{prop: approach}(i).
Thus, it suffices to show $g_{T_k +s}E'x_0 \subset g_{[-t,t]}K$ for all $0 \leq s \leq t_{k+1}$. For such an $s$, we compute
\begin{equation*}
   % \begin{split}
        g_{T_k+s} E'x_0 %&
        = g_{T_k+s} \Phi_{-T_{k+1}}(\overline{V}\gamma)p_k(E)x_0 %\\
        %&
        = g_{s} \Phi_{-t_{k+1}}(\overline{V}\gamma)g_{T_{k}}p_k(E)x_0.
   % \end{split}
\end{equation*}
Now if $k+1 \notin f(\N)$, we know that $t_{k+1} = t$, and the conclusion follows from Proposition \ref{prop: generic}. Otherwise, %if $k+1 = f(n)$ 
 in view of \eqref{deflambdaxk},  \eqref{eq: children inductive} and Proposition \ref{prop: approach}, for some $n \in \N$ we have
\begin{equation}\label{approachdetailed}
   \begin{split}
 \ca{E}_{k+1}(E) = \{{E'}\},\   t_{k+1} = s_{\varepsilon_n}&, \text{ and }\\
        g_{T_k+s} {E'}x_0 %&
        = g_{s} \Phi_{-s_{\varepsilon_n}}(\overline{V}\gamma)g_{T_{k}}p_k(E)x_0 &= g_{s} \overline{W}g_{T_{k}}p_k(E)x_0,\\ \text{ where }W= \Phi_{-s_{\varepsilon_n}}(V \gamma) \text{ as in Pro}&\text{position \ref{prop: approach}(i)}. 
   \end{split}
\end{equation}
Hence the conclusion follows from Proposition \ref{prop: approach}(ii).
%We now come to 
Finally, to demonstrate (c) it suffices to only consider the  approach step case.
%. Suppose 
%$k+1 = f(n)$ for some $n \in \N$. 
Then one can use \eqref{approachdetailed} and Proposition \ref{prop: approach}(iii) to conclude that $$g_{T_k+s'_{\varepsilon_n}} {E'}x_0 %&
        =  g_{s'_{\varepsilon_n}} \overline{W}g_{T_{k}}p_k(E)x_0$$
lies in the $\varepsilon_n$-neighborhood of $B_Z(\sigma)\beta(n)$.

\smallskip
This completes  the inductive construction. As before, it is easy to show that the family of sets $\ca{E}_k$ satisfies the
%first five} 
requirements of Definition \ref{defn: tree like collection}; in particular, in view of  \equ{diamV} and \eqref{eq: lambda min}, \begin{equation}\label{dkgeneral}
    d_k = \sup\{\on{diam}_H(E): E \in \ca{E}_k\} \leq e^{\frac{-\lambda_{\on{min}}T_k}{2}}.
\end{equation}
%is complete and, i
In light of (b), (c) and the infinitude of the fibers of $\beta$,
%as well as the assumption on $\sigma$, 
 the set in \eqref{eq: da one and only} contains $E_\infty x_0$.
%Note that by what is proved in the induction above, it is easy to show that the family of sets $\ca{E}_k$ satisfies the requirements of Definition \ref{defn: tree like collection}.

It thus remains to estimate the Hausdorff dimension of %the latter
$E_\infty$. Clearly $$\Delta_k = \inf\{\on{density}(\ca{E}_{k+1},E): E \in \ca{E}_k\}$$ depends on whether   $k+1$ corresponds to a generic or a special step. In the former case estimate \eqref{eq: densitybound} still holds. Otherwise  $k+1 =f(n)$ for some $n\in\N$,  $\ca{E}_{k+1}(E)$ is a singleton for every $E \in \ca{E}_k$, and $\Delta_k = e^{-\chi s_{\varepsilon_n}}$. Hence, analogously to \eqref{tk}, we can write \begin{equation}\label{deltak}
\begin{aligned}
    \sum_{j=0}^k\log \Delta_j  &\geq \big(k+1-\#f^{-1}(\{1,\dots,k+1\})\big) \log \big(m_X(K)^2\big) - \chi \sum_{f(n)\le k+1} s_{\varepsilon_n}\\ &= 2(k+1)  \log m_X(K) - \sum_{f(n)\le k+1} \big(\chi s_{\varepsilon_n}+2\log m_X(K)\big).
    \end{aligned}
\end{equation}
Now, using \eqref{tk}, \eqref{dkgeneral} and \eqref{deltak} we can apply Theorem \ref{thm: urbanski} %for any $x\in K$ 
to obtain
\begin{equation*}
\begin{aligned}
        \dim H - \dim  {E}_\infty \leq \ &\limsup_{k \to \infty} \frac{\sum_{j=0}^{k} \log \Delta_j}{\log d_k}\\
        \leq \ &\limsup_{k \to \infty} \frac{2(k+1)  \log m_X(K) - \sum_{f(n)\le k+1} \big(\chi s_{\varepsilon_n}+2\log m_X(K)\big)}{-\frac12{\lambda_{\on{min}}\big(k  t + \sum_{f(n)\le k}(s_{\varepsilon_n}-t) \big)}} \\ = \ &\limsup_{k \to \infty} \frac{2(k+1)  \log  \frac1{m_X(K) }  + \sum_{f(n)\le k+1} \big(\chi s_{\varepsilon_n}-2\log \frac1{m_X(K) }\big)}{\frac12{\lambda_{\on{min}}\big(k  t + \sum_{f(n)\le k}(s_{\varepsilon_n}-t) \big)}}\\ \le\  &
        \frac{4  \log  \frac1{m_X(K) }} {\lambda_{\on{min}}   t} + \frac{2\chi }{\lambda_{\on{min}}   t}
        \limsup_{k \to \infty} \frac{  \sum_{f(n)\le k+1}   s_{\varepsilon_n} }k .
       \end{aligned}
    \end{equation*}
    But since the sequence $\{f(n)\}$ can be fixed after the choice of the values $s_{\varepsilon_n}$ is made, it simply remains to choose $f$ sparse enough so that $\sum_{f(n)\le k+1}   s_{\varepsilon_n}$ grows sublinearly in $k$. Hence, as in the proof for $B = \varnothing$,  we get $$\dim  {E}_\infty\ge \dim H - \frac{4  \log  \frac1{m_X(K) }} {\lambda_{\on{min}}   t},$$ and letting $t\to\infty$ we conclude that the set \eqref{eq: da one and only} is thick.
    \end{proof}


\begin{thebibliography}{BGMRV\ }
 \small

 %\bibitem[AW]{AW} {A.\ Agin and B.\ Weiss}, \textit{The Dirichlet spectrum}, {\tt https://arxiv.org/abs/2412.05858}, (2024).

 \bibitem[AD]{AD} N.\ Andersen and W.\ Duke, \textit{On a theorem of Davenport and Schmidt}, Acta Arith. {\bf 198} (2021), no.\ 1,  37--75.
  
 
 \bibitem[AGK]{AGK} {J.\ An, L.\ Guan and D.\ Kleinbock}, \textit{Nondense orbits on homogeneous spaces and applications to geometry and number theory}, %preprint (2020), {\tt https://urldefense.com/v3/__https://arxiv.org/pdf/2001.05174.pdf__;!!DaRZpAeNFA!bnT1fRDPUohUmHXMC2YZUHQgLf_CeQG9GtUB_6hc2kHWyStb_c0rSBiIm9MkXkFua9mu1QEcY_5D_z9Iv8rOl7Ue$ [arxiv[.]org]}, 
   % to appear in 
    {Ergodic Theory Dynam. Systems} %(2021), DOI: {\tt https://urldefense.com/v3/__https://doi.org/10.1017/etds.2021.4__;!!DaRZpAeNFA!bnT1fRDPUohUmHXMC2YZUHQgLf_CeQG9GtUB_6hc2kHWyStb_c0rSBiIm9MkXkFua9mu1QEcY_5D_z9Iv1-Kkj45$ [doi[.]org]}.  
     {\bf 42} (2022), no.\ 4, 1327--1372.

 \bibitem[B]{B} N. Bourbaki, \textit{Integration. II. Chapters 7–9},
Translated from the 1963 and 1969 French originals by Sterling K.\ Berberian,
Elem.\ Math.\ (Berlin),
Springer--Verlag, Berlin, 2004.


   %  \bibitem[BP]{BiPe} C.\ Bishop, Y.\ Peres, \textit{Fractals in Probability and Analysis}, Cambridge University Press, 2016.

    

\bibitem[BFK]{BFK} R.\ Broderick, L.\ Fishman, D.\ Kleinbock, \textit{Schmidt's game, fractals, and orbits of toral endomorphisms},  Ergodic Theory Dynam.\ Systems {\bf 31}  (2011),  no.\ 4, 1095--1107.

%     \bibitem[BGMRV]{BGMRV}  V. Beresnevich, L. Guan, A. Marnat, F. Ramirez, S. Velani, \textit{Dirichlet is not just Bad and Singular},    Adv. Math. 401 (2022), Paper No. 108316, 57 pp.

%      \bibitem[BRV]{BRV} V.\ Beresnevich, F.\ Ramirez and S.\ Velani, \textit{Metric Diophantine Approximation: Aspects of Recent Work}, in: Dynamics and Analytic Number Theory, LMS Lecture Note Series {\bf 437} (2016).

     %\bibitem[BC]{BC} A.\ Borel, Harish Chandra, \textit{Arithmetic Subgroups of Algebraic Groups}, Annals of Mathematics, Second Series  {\bf 75}, No. 3 (1962), pp. 485-535.

     \bibitem[BM]{BM} J.\ Brezin and C.\ Moore, \textit{Flows on homogeneous spaces: A new look}, Amer. J.
Math. {\bf 103} (1981), 571--613.

%     \bibitem[C]{C} N.\ Chevallier, \textit{Gauss lattices and complex continued fractions}, Pure Appl. Math. Q. 17 (2022), no. 5, 1785--1860.


\bibitem[D1]{Da1} S.\,G.\ Dani,  \textit{Bounded orbits of flows on \hs s}, Comment.\ Math.\ Helv.\ {\bf 61} (1986), 636--660.

    
    
\bibitem[D2]{Da3} \bysame,  \textit{On orbits of endomorphisms of tori and the Schmidt game}, Ergodic Theory Dynam.\ Systems {\bf 8} (1988), 523--529.

\bibitem[Do]{Do}  D.\ Dolgopyat, \textit{Bounded orbits of Anosov flows},  Duke Math.\ J.\ {\bf 87} (1997), no.\ 1, 87--114.



\bibitem[F]{F}   K.\,J.\ Falconer, \textit{The geometry of fractal sets},
Cambridge Tracts in Math.\ {\bf 85},
 Cambridge Univ.\ Press, 
 Cambridge and New York,
1986.

%\bibitem[DS]{DS}  H.\ Davenport and W.\,M.\ Schmidt, \textit{Dirichlet's theorem on diophantine approximation}, in: 1970 Symposia Mathematica, Vol.\ IV (INDAM, Rome, 1968/69),  pp.\ 113--132, Academic Press, London.

%\bibitem[EW]{EW} M. Einsiedler, T. Ward, \textit{Ergodic theory with a view towards number theory}, {Graduate Texts in Mathematics},    {259}, {Springer-Verlag London Ltd.},{2011}, {xviii+481}.

\bibitem[EH]{EH} B.\ Esmayli and P.\ Haj{\l}asz, \textit{The coarea inequality}, Annales Fennici Mathematici
{\bf 46} (2021),  965–-991.

%\bibitem[Fe]{Fe} H.\ Federer, \textit{Geometric measure theory}, Classics in Mathematics, Springer-Verlag Berlin Heidelberg (1996).


%\bibitem[Fo]{Fo} G.\ Folland, \textit{A course in abstract harmonic analysis}, 2ed, Textbooks in Mathematics, CRC Press (2016).


%\bibitem[GP]{GP} D.\ Gatzouras, and Y.\ Peres, \textit{Invariant measures of full dimension for some expanding maps}, Ergodic Theory Dynam.\ Systems {\bf 17} (1997), no.\ 1, 147--167.

% \bibitem[H]{H}   {G.\ Haj\'{o}s}, \textit{\"{U}ber einfache und mehrfache {B}edeckung des       {$n$}-dimensionalen {R}aumes mit einem {W}\"{u}rfelgitter}, {Math. Z.} {\bf 47} ({1941}), {427--467}.

%   \bibitem[HSW]{HSW} M.\ Hussain, J.\ Schleischitz, B.\ Ward, \textit{On the Folklore set and Dirichlet spectrum for matrices}, {\tt https://arxiv.org/abs/2402.13451}, (2024).

%\bibitem[Ka]{Ka}  
%A.\ K\"aenm\"aki,  \textit{Measures of full dimension on self-affine sets}, 
%Acta Univ.\ Carolin.\ Math.\ Phys.\ {\bf 45} (2004), no.\ 2, 45--53.

\bibitem[Kl1]{Kl2} D.\ Kleinbock, \textit{Nondense orbits of flows on homogeneous spaces}, Ergodic Theory Dynam.\ Systems {\bf 18} (1998), 373--396.

\bibitem[Kl2]{Kl} \bysame, \textit{Badly approximable systems of affine forms}, J. Number Theory {\bf 79} (1999), pp. 83--102.   

 \bibitem[KMa]{KM} {D.\  Kleinbock and G.\,A.\ Margulis},
     \textit{Bounded orbits of nonquasiunipotent flows on homogeneous spaces},
{Sina\u{\i}'s {M}oscow {S}eminar on {D}ynamical {S}ystems},
    {Amer. Math. Soc. Transl. Ser. 2} \textbf{171}, {pp. 141--172}, Amer. Math. Soc., Providence, RI, 1996. 
     % {(1996)}.

% \bibitem[KMi]{KMi}D.\   Kleinbock and S.\ Mirzadeh, \textit{Dimension drop for diagonalizable flows on homogeneous spaces},  J.\ Mod.\ Dyn.\ {\bf 20} (2024), 441–-478.


    \bibitem[Kn]{Kn} A.\,W.\ Knapp, \textit{Lie Groups Beyond an Introduction} 2ed, Progress in Mathematics, Birkhäuser Boston, MA (2002).

  %   \bibitem[KR1]{KR1} D.\ Kleinbock  and  A.\ Rao, \textit{Weighted uniform Diophantine approximation of systems of linear forms},   Pure Appl. Math. Q. {\bf 18} (2022), no. 3, 1095--1112.

%\bibitem[KR2]{KR2} D.\ Kleinbock and A.\ Rao, \textit{Abundance of Dirichlet-improvable pairs with respect to arbitrary norms} Mosc. J. Comb. Number Theory, \textbf{11}, 97-114 (2022).

 \bibitem[KR1]{KR} D.\ Kleinbock and  A.\ Rao, \textit{A zero-one law for uniform Diophantine approximation in Euclidean norm},  Internat.\ Math.\ Res.\ Notices {\bf 2022}, no.\ 8, 5617--5657.

 \bibitem[KR2]{KR1} \bysame, \textit{A dichotomy phenomenon for bad minus normed Dirichlet}, Mathematika {\bf 69} (2023), no.\  4, 1145--1164.  


\bibitem[KRS]{KRS}
D.\ Kleinbock, A.\ Rao and S.\ Sathiamurthy, \textit{Critical loci of convex domains in the plane}, %arXiv preprint {\tt https://arxiv.org/pdf/2003.13829.pdf} (2020), to appear in 
Indag.\ Math.\  (N.S.) {\bf 32} (2021), no.\ 3, 719--728. 


 %\bibitem[KW1]{KW1} D.\ Kleinbock, B.\ Weiss, \textit{Modified Schmidt games and a conjecture of Margulis}, J. Mod. Dyn. 7, no. 3 (2013), 429–460.

% \bibitem[KW2]{KW2} D. Kleinbock, B. Weiss, \textit{Dirichlet’s theorem on diophantine approximation and homogeneous flows}, J. Mod. Dyn. 4 (2008), 43–62.

% \bibitem[L]{L}  N.\ Luzia, \textit{Measure of full dimension for some nonconformal repellers}, Discrete Contin.\ Dyn.\ Syst.\ {\bf 26} (2010), no.\ 1, 291--302.


%\bibitem[Mc]{Mc} C.\ McMullen, \textit{Minkowski's conjecture, well-rounded lattices and topological dimension}, 
% {Journal of the American Mathematical Society},
%    \textbf{18},
%      {2005},
  %   {711--734}.
%
\bibitem[M]{M} G.\,A.\ Margulis, \textit{On Some Aspects of the Theory of Anosov Systems}, Springer Monograph in Mathematics, Springer, 2004.
%\bibitem[Mc]{Mc} C.\ McMullen, \textit{Minkowski's conjecture, well-rounded lattices and topological dimension}, 
% {Journal of the American Mathematical Society},
%    \textbf{18},
%      {2005},
    % {711--734}.
%

\bibitem[Ma]{Ma} J.\,M.\ Marstrand, \textit{The dimension of Cartesian product sets} Proc.\ Camb.\ Phil.\ Soc.\ {\bf 50} (1954), 198--202.

%P.\ Mattila, \textit{Geometry of Sets and Measures in Euclidean Spaces: Fractals and Rectifiability}, Cambridge University Press; 1995. 

%\bibitem[Mi]{Mi} H.\ Minkowski, \textit{Diophantische Approximationen}, Mathematische Vorlesungen an der Universität Göttingen, Teubner Verlag Wiesbaden, Springer Fachmedien Wiesbaden 1907.

     \bibitem[Mo]{Mo} C.\ C.\ Moore, \textit{Ergodicity of flows on homogeneous spaces}, Amer.\ J.\  Math. {\bf  88}  (1966), no.~1, 154--187.


     \bibitem[MS]{MS} N.\ Moshchevitin and N.\ Shulga, \textit{Dirichlet improvability in $L_p$-norms}, preprint (2024), {\tt arXiv:2408.06200}.



     \bibitem[S]{S} {W.\,M.\ Schmidt},    \textit{On badly approximable numbers and certain games}, {Trans.\ Amer.\ Math.\ Soc.}\ {\bf 123} (1966), 178--199.

%\bibitem[S2]{Schmidt} \bysame, \textit{\da},  Lecture Notes in Mathematics, vol.\ 785, Springer-Verlag, Berlin, 1980.

 %    \bibitem[T]{T} A.\ Terras,   \textit{Harmonic Analysis on Symmetric Spaces and Applications II}, Second Edition, Springer Science+Business Media, New York (2016).


\bibitem[Ts1]{Ts1} J.\ Tseng,
\textit{Schmidt games and Markov partitions},
Nonlinearity {\bf 22} (2009), no.\ 3, 525--543.

\bibitem[Ts2] {Ts2}  \bysame, 
 \textit{Nondense orbits for Anosov diffeomorphisms of the $2$-torus}, Real Anal.\ Exchange {\bf 41} (2016), no.\ 2, 307--314.



     \bibitem[U]{U} M.\ Urbanski, \textit{The Hausdorff dimension of the set of points with
nondense orbit under a hyperbolic dynamical
system}, Nonlinearity {\bf 2} (1991), 385--397.


\bibitem[Wu1]{wu1} W.\ Wu, \textit{Schmidt games and non-dense forward orbits of certain partially hyperbolic systems}, Ergodic Theory Dynam.\ Systems {\bf 36} (2016), no.\ 5, 1656--1678.

\bibitem[Wu2]{wu2} \bysame,  \textit{Modified Schmidt games and non-dense forward orbits of partially hyperbolic systems},  Discrete Contin.\ Dyn.\ Syst.\ {\bf 36} (2016), no.\ 6, 3463--3481.


%\bibitem[Y]{Y} Y.\ Yayama, \textit{Dimensions of compact invariant sets of some expanding maps}, Ergodic Theory Dynam.\ Systems {\bf 29} (2009), no.\ 1, 281--315.

   \end{thebibliography}
\end{document}